\documentclass{amsart}

\usepackage{amsmath,amssymb,amsthm, mathrsfs, mathtools, bm, amsfonts}
\usepackage{mathabx}
\usepackage{amscd,mathtools}
\usepackage[utf8]{inputenc}
\usepackage[english]{babel}
\usepackage{cite}
\usepackage{color}
\usepackage[pagebackref,colorlinks,citecolor=blue,linkcolor=red]{hyperref}

\usepackage{float}
\usepackage{tikz}
\usepackage{lmodern}
\usetikzlibrary{decorations.pathmorphing}

\usepackage{multicol}
\usetikzlibrary{cd}
\usepackage{geometry}

\usepackage{comment}
\usepackage{mathtools}

\usepackage{csquotes}

\usepackage{amsthm}

\usepackage{cleveref}

\usepackage{longtable}

\newtheorem{theorem}{Theorem}[section]
  \newtheorem{proposition}[theorem]{Proposition}

\newtheorem{corollary}[theorem]{Corollary}
\newtheorem{lemma}[theorem]{Lemma}
\newtheorem{claim}[theorem]{Claim}

\theoremstyle{definition}
\newtheorem{defn}{Definition}

\newtheorem{remark}[theorem]{Remark}
\newtheorem{example}[theorem]{Example}

\usepackage{xcolor}
\usepackage{fancyvrb}
\definecolor{Green}{RGB}{0,180,60}
\definecolor{Purple}{RGB}{150,0,250}

\usepackage{listings}

\usepackage{xparse}

\NewDocumentCommand{\gap}{v}{%
\textbf{\texttt{\textcolor{blue}{#1}}}%
}
\NewDocumentCommand{\codeinput}{v}{%
\texttt{\textcolor{red}{#1}%
}}
\NewDocumentCommand{\codeoutput}{v}{\texttt{\textcolor{black}{#1}}}
\NewDocumentCommand{\codecomment}{v}{\texttt{\textcolor{purple}{#1}}}
\usepackage{array} 

\usepackage{soul} 

\title{Balanced Spanning Trees for Triangular Strip Lattices}

\author{Saugat Dhakal}
\address{Rice University}
\email{sd200@rice.edu}

\author{Long Nguyen}
\address{Rice University}
\email{ln50@rice.edu}

\author{Yandi Wu}
\address{Rice University}
\email{yandi.wu@rice.edu}

\begin{document}

\maketitle
\begin{abstract}
    A balanced spanning tree is a spanning tree that contains an edge whose removal partitions the vertices into exactly two connected subtrees of equal size. In this paper, we establish explicit recurrence relations for the number of spanning trees in $2 \times n$ \textit{triangular strip lattices}—obtained by adding a diagonal edge to each square of a $2 \times n$ grid graph—generalizing combinatorial counting techniques introduced by Raff~\cite{raff}.  We then adapt arguments of Gallagher and Tapp \cite{gal25} to count balanced spanning trees of arbitrary triangular strip lattices. We establish sharp asymptotic bounds for the proportion of balanced spanning trees as $n \to \infty$. Finally, we determine the probability that a spanning tree of a $2\times n$ triangular strip lattice chosen uniformly at random is balanced as $n \to \infty$.
\end{abstract}




\section{Introduction}

\emph{Balanced spanning trees} (Definition \ref{defn: Balanced_spanning_tree}) have recently attracted attention because of their applications to computational redistricting. One of the most prominent redistricting algorithms is ReCom, which has been observed empirically to generate districting plans with compact districts \cite{deford2021recombination,clelland2021compactness}. More recently, Procaccia and Tucker-Foltz \cite{procaccia2022compact} provided a theoretical explanation for this phenomenon by showing that the spanning-tree-based sampling procedure used by ReCom favors redistricting plans with shorter boundaries, which are more compact. This connection between spanning trees and compact districts is rooted in the way ReCom constructs new districting plans. 

Specifically, ReCom operates on the dual graph of a geographic region, where each vertex represents a geographic unit (such as a census block, precinct, or county) and edges connect adjacent geographic units (see Figure \ref{fig:wvdual}). At each iteration, the algorithm temporarily merges two adjacent districts, samples a spanning tree of the corresponding subgraph, and removes an edge whose deletion yields two connected components with nearly equal populations; that is, it seeks balanced spanning trees in the sense of Definition~\ref{defn: Balanced_spanning_tree} \cite{deford2021recombination}. The two components then become the new districts. Consequently, understanding the prevalence of balanced spanning trees may provide theoretical insight into the behavior of ReCom and related spanning-tree-based redistricting algorithms.

\begin{figure}
    \centering    \includegraphics[width=0.5\linewidth]{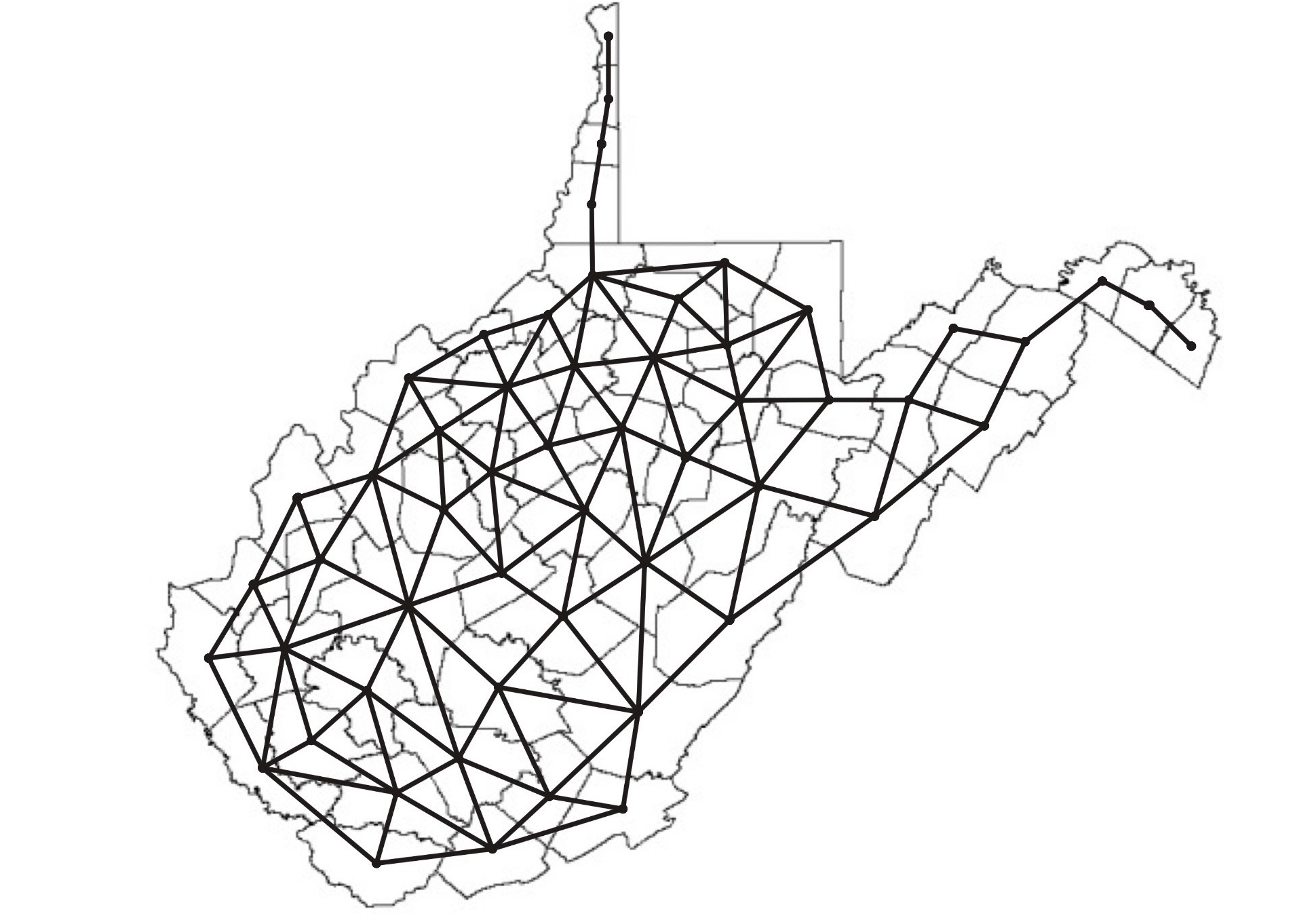}
    \caption{Dual graph for West Virginia where the geographical units are counties.} 
    \label{fig:wvdual}
\end{figure}

As illustrated in Figure \ref{fig:wvdual}, dual graphs arising from real geographic regions often contain square and triangular faces. This observation has motivated several recent studies of spanning trees and balanced partitions in families of graphs where the faces are all squares or triangles. Cannon, Pegden, and Tucker-Foltz \cite{cannon2024balancedforests} studied balanced partitions of grid graphs and developed polynomial-time algorithms for sampling balanced forests, establishing bounds on the probability that a random spanning tree admits a balanced partition. Gallagher and Tapp \cite{gal25} later derived an exact expression for the proportion of balanced spanning trees in the $2 \times n$ grid graph. Wang \cite{wang2023triangular} investigated spanning tree enumeration on three families of graphs with triangular faces—rows of triangles, fan graphs, and chains of triangles. More recently, Guglani \cite{guglani2025triangular} established a lower bound for the proportion of balanced spanning trees in hexagonal regions of triangular lattices.

In this paper, we study balanced spanning trees in rows of triangles. 
In particular, we consider \emph{$2\times n$ triangular strip lattices}, obtained by adding a diagonal to each square face of a $2\times n$ grid graph, and establish sharp asymptotic bounds for the proportion of balanced spanning trees as $n\to \infty.$ We further determine the limiting probability that a uniformly random spanning tree of a $2 \times n$ triangular strip lattice is balanced.

\begin{figure}[H]
    \centering
    \begin{tikzpicture}[thick]
        \draw grid(6,1);
        \draw (0,1)--(1,0)--(2,1) (2,0)--(3,1)--(4,0) (4,1)--(5,0)--(6,1);
        \useasboundingbox (0,-1) rectangle (6,2);
    \end{tikzpicture}
    \caption{An example of a $2$-by-$7$ triangular strip lattice in $\mathscr{G}_{14}$.}
   \label{fig:trianglestriplattice}
\end{figure}
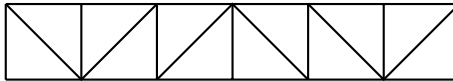

Before presenting our main theorems, we formally define a balanced spanning tree.

\begin{defn}[Balanced spanning tree]\label{defn: Balanced_spanning_tree}
Let $\mathcal{T}$ be a spanning tree of a graph $G = (V, E)$ with an even number of vertices, $|V| = 2n$. We say that $\mathcal{T}$ is a \emph{balanced spanning tree} if there exists an edge $e \in E(\mathcal{T})$ such that the forest $\mathcal{T} \setminus \{e\}$ consists of exactly two disjoint, connected subtrees, each spanning exactly $n$ vertices in $G$. 
\end{defn}

Our main results are summarized in the following two theorems.
\begin{theorem}[Asymptotic Bounds for Balanced Spanning Trees]\label{thm:maintheorem} Let $\mathscr{G}_{2n}$ be the family of triangular strip lattices with $2n$ vertices, where $n \in \mathbb{N}$. Then for any $\Gamma_{2n} \in \mathscr{G}_{2n}$, the asymptotic proportion of balanced spanning trees satisfies the following bounds:
\begin{align*}
\text{Even Case ($n = 2m$): } &\quad 0.63087 \leq \lim_{n \to \infty} \dfrac{S_{2n}}{T_{2n}} \leq 0.668328, \\[0.5em]
\text{Odd Case ($n = 2m + 1$): } &\quad 0.544933 \leq \lim_{n \to \infty} \dfrac{S_{2n}}{T_{2n}} \leq 0.668328.
\end{align*}
where $T_{2n}$ and $S_{2n}$ denote the number of spanning trees and balanced spanning trees respectively in $\Gamma_{2n}$.
\end{theorem} 

While \Cref{thm:maintheorem} provides sharp asymptotic upper and lower bounds, we are also interested in the behavior of a typical triangular strip lattice. To model such a lattice, we assign each diagonal edge one of the two possible orientations independently and with equal probability (see \Cref{sec:proportion} for a more precise description), where the orientation specifies whether the diagonal joins the top-left and bottom-right vertices or the top-right and bottom-left vertices of its corresponding square face. We then calculate the asymptotic proportion of balanced spanning trees for this random triangular strip lattice.

\begin{theorem}[Average Proportion for Typical Triangular Strip Lattices]\label{thm:expectedvalue}
Let $\Gamma_{2n}$ be a triangular strip lattice whose diagonal configuration is chosen uniformly at random. As $n \to \infty$, the expected proportion of balanced spanning trees converges to the mean of the asymptotic bounds:
\begin{align*}
\text{Even Case $(n = 2m)$: } \quad &\lim_{n \to \infty} \mathbb{E}\left[\frac{S_{2n}}{T_{2n}}\right] \approx 0.6496, \\[0.5em]
\text{Odd Case $(n = 2m + 1)$: } \quad &\lim_{n \to \infty} \mathbb{E}\left[\frac{S_{2n}}{T_{2n}}\right] \approx 0.6066.
\end{align*}
\end{theorem}

\section*{Acknowledgments}

We gratefully acknowledge the support provided by the Summer Undergraduate Research Fellowship (SURF) at Rice University, which provided funding and made this project possible.

\section{Counting Spanning Trees}

In this section, we count the number of spanning trees in a given triangular strip lattice. Our approach is inspired by Raff's method for enumerating spanning trees in grid graphs \cite{raff}.

\subsection{Counting spanning trees of a triangular strip lattice with an even number of vertices}\label{thm:spanningtrees}
We begin by considering triangular strip lattices with an even number of vertices, as balanced spanning trees are defined only for graphs with an even number of vertices. Let $\Gamma_{2n}$ denote an arbitrary triangular strip lattice with $2n$ vertices, where $n \in \mathbb{N}.$ Let $T_{2n}$ denote the number of spanning trees of $\Gamma_{2n}.$

\begin{theorem}[Number of spanning trees in $\Gamma_{2n}$]\label{thm:spanningtrees}
For any triangular strip lattice $\Gamma_{2n},$ the number of spanning trees satisfies the recurrence
\[T_{2(n+2)} = 7T_{2(n+1)} - T_{2n},\]
with initial conditions $T_2=1$ and $T_4=8.$
\end{theorem}

\begin{proof}
To derive a recurrence for the number of spanning trees, we will examine all ways a spanning tree in $\Gamma_{2(n+1)}$ can be constructed by adding two vertices to the right end of $\Gamma_{2n}$.\\
\indent Let $u$ and $v$ denote the two rightmost vertices of $\Gamma_{2(n+1)}$, and let $\mathcal{T}$ be a spanning tree of $\Gamma_{2(n+1)}.$ Since $\mathcal{T}$ is a spanning tree of $\Gamma_{2(n+1)},$ the graph $\mathcal{T} \setminus \{u,v\}$ contains no loops and spans all vertices of $\Gamma_{2n}.$ Consequently, $\mathcal{T}\setminus\{u,v\}$ is either
\begin{enumerate}
    \item a spanning tree of $\Gamma_{2n}$ or
    \item a spanning forest of $\Gamma_{2n}$ consisting of two components, with the two rightmost vertices of $\Gamma_{2n}$ lying in different trees.
\end{enumerate}

Note that it suffices to consider forests from (2) because if there are more than two trees in the forest or if the rightmost vertices of the forest lie on the same tree, the added vertices cannot adjoin the trees into a single spanning tree.

\indent Accordingly, let $\mathscr{T}_{2n}$ denote the set of spanning trees in $\Gamma_{2n},$ and let $\mathscr{F}_{2n}$ denote the set of spanning forests in $\Gamma_{2n}$ described in (2). Starting from a spanning tree in $\mathscr{T}_{2n}$, we can extend to a spanning tree in $\mathscr{T}_{2(n+1)}$ by adding two new vertices $u$ and $v$ to its right end. For this to happen, we must pick edges that connect both $u$ and $v$ to the rest of the tree without creating any loops. Note that if one were to pick a diagonal, that diagonal must either join the top-left and bottom-right vertices or the top-right and bottom-left vertices as predetermined by the structure of $\Gamma_{2n}.$ For either type of diagonal, there are five possible ways to extend a spanning tree in $\mathscr{T}_{2n}$ to a spanning tree in $\mathscr{T}_{2(n+1)}.$

\begin{figure}[H]
    \centering
        \begin{tikzpicture}[thick, every node/.style={font=\sffamily\large}, xscale=0.8,yscale=0.8]
    \fill[lightgray] (0, 2) rectangle (3, 3);
    \node at (1.5, 2.5) {\textcolor{white}{Tree}};
    
    \fill (3, 3) circle (1.5pt);
    \fill (3, 2) circle (1.5pt);
    \fill (4, 3) circle (1.5pt);
    \draw (3, 3) -- (4, 3);
    \draw (3, 2) -- (4, 2);
    \fill (4, 2) circle (1.5pt);
    
    \fill[lightgray] (5, 2) rectangle (8, 3);
    \node at (6.5, 2.5) {\textcolor{white}{Tree}};
    
    \fill (8, 3) circle (1.5pt);
    \fill (8, 2) circle (1.5pt);
    \fill (9, 3) circle (1.5pt);
    \fill (9, 2) circle (1.5pt);
    \draw (8, 3) -- (9, 3);
    \draw (9, 2) -- (9, 3);

    \fill[lightgray] (10, 2) rectangle (13, 3);
    \node at (11.5, 2.5) {\textcolor{white}{Tree}};
    
    \fill (13, 3) circle (1.5pt);
    \fill (13, 2) circle (1.5pt);
    \fill (14, 3) circle (1.5pt);
    \fill (14, 2) circle (1.5pt);

    \draw (13, 2) -- (14, 2);
    \draw (14, 2) -- (14, 3);

    \fill[lightgray] (2.5, 0) rectangle (5.5, 1);
    \node at (4, 0.5) {\textcolor{white}{Tree}};
    \fill (5.5, 1) circle (1.5pt);
    \fill (5.5, 0) circle (1.5pt);
    
    \draw (5.5, 1) -- (6.5, 1);
    \draw (5.5, 1) -- (6.5, 0);
    \fill (6.5, 0) circle (1.5pt);
    \fill (6.5, 1) circle (1.5pt);

    \fill[lightgray] (7.5, 0) rectangle (10.5, 1);
    \node at (9, 0.5) {\textcolor{white}{Tree}};
    \fill (10.5, 1) circle (1.5pt);
    \fill (10.5, 0) circle (1.5pt);
    
    \draw (10.5, 1) -- (11.5, 0);
    \draw (11.5, 0) -- (11.5, 1);
    \fill (11.5, 0) circle (1.5pt);
    \fill (11.5, 1) circle (1.5pt);

\end{tikzpicture}
    \caption{Possible ways to extend a spanning tree in $\mathscr{T}_{2n}$ to a spanning tree in $\mathscr{T}_{2(n+1)}.$}
    \label{fig:extendtrees}
\end{figure}
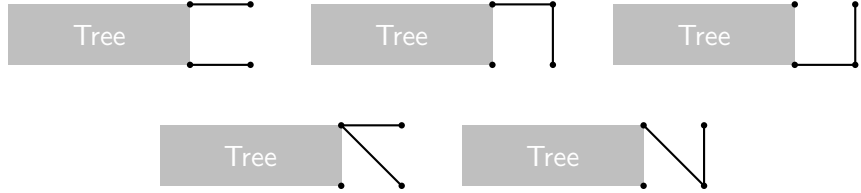

The only structure that can extend to a tree in $\mathscr{T}_{2(n+1)}$ in addition to spanning trees is a forest in $\mathscr{F}_{2n}$. There are three ways each forest in $\mathscr{F}_{2n}$ can be extended to a tree in $\mathscr{T}_{2(n+1)}$ (See Figure \ref{fig:foreststotree}). 

\begin{figure}[H]
    \centering
    \begin{tikzpicture}[thick, every node/.style={font=\sffamily\large},xscale=0.8,yscale=0.8]
    \fill[lightgray] (0, 2) rectangle (3, 3);
    \node at (1.5, 2.5) {\textcolor{white}{Forest}};
    
    \fill (3, 3) circle (1.5pt);
    \fill (3, 2) circle (1.5pt);
    \fill (4, 3) circle (1.5pt);
    \draw (3, 3) -- (4, 3);
    \draw (3, 2) -- (4, 2);
    \draw (4, 2) -- (4, 3);
    \fill (4, 2) circle (1.5pt);
    
    \fill[lightgray] (5, 2) rectangle (8, 3);
    \node at (6.5, 2.5) {\textcolor{white}{Forest}};
    
    \fill (8, 3) circle (1.5pt);
    \fill (8, 2) circle (1.5pt);
    \fill (9, 3) circle (1.5pt);
    
    \draw (8, 3) -- (9, 2);
    \draw (8, 2) -- (9, 2);
    \draw (8, 3) -- (9, 3);
    \fill (9, 2) circle (1.5pt);

    \fill[lightgray] (2.5, 0) rectangle (5.5, 1);
    \node at (4, 0.5) {\textcolor{white}{Forest}};
    \fill (5.5, 1) circle (1.5pt);
    \fill (6.5, 0) circle (1.5pt);
    
    \draw (5.5, 0) -- (6.5, 0);
    \draw (5.5, 1) -- (6.5, 0);
    \draw (6.5, 0) -- (6.5, 1);
    \fill (5.5, 0) circle (1.5pt);
    \fill (6.5, 1) circle (1.5pt);

\end{tikzpicture}
    \caption{Possible ways to extend a forest in $\mathscr{F}_{2n}$ to a spanning tree in $\mathscr{T}_{2(n+1)}$}
    \label{fig:foreststotree}
\end{figure}
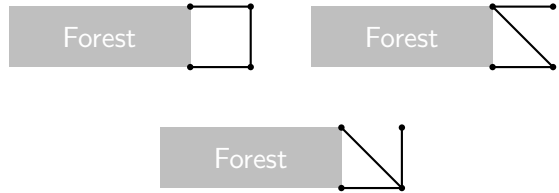

Let $T_k = |\mathscr{T}_k|$ and $F_k = |\mathscr{F}_k|$. From the previous paragraphs, we obtain the recurrence
\[T_{2(n+1)} = 5T_{2n} + 3F_{2n}.\]
Since $F_{2n}$ is featured in the above relation, it is equally important to keep track of $F_{2n}$. We can derive a recurrence for $F_{2n}$ by similar reasoning. Starting from a spanning tree in $\mathscr{T}_{2n}$, there are three ways to append two vertices to get a forest in $\mathscr{F}_{2(n+1)}$: 

\begin{figure}[H]
    \centering
        \begin{tikzpicture}[thick, every node/.style={font=\sffamily\large},xscale=0.8,yscale=0.8]
    \fill[lightgray] (0, 2) rectangle (3, 3);
    \node at (1.5, 2.5) {\textcolor{white}{Tree}};
    
    \fill (3, 3) circle (1.5pt);
    \fill (3, 2) circle (1.5pt);
    \fill (4, 3) circle (1.5pt);
    \draw (4, 3) -- (3, 3);
    \fill (4, 2) circle (1.5pt);
    
    \fill[lightgray] (5, 2) rectangle (8, 3);
    \node at (6.5, 2.5) {\textcolor{white}{Tree}};
    
    \fill (8, 3) circle (1.5pt);
    \fill (8, 2) circle (1.5pt);
    \fill (9, 3) circle (1.5pt);
    
    \draw (8, 3) -- (9, 2);
    \fill (9, 2) circle (1.5pt);

    \fill[lightgray] (2.5, 0) rectangle (5.5, 1);
    \node at (4, 0.5) {\textcolor{white}{Tree}};
    \fill (5.5, 1) circle (1.5pt);
    \fill (5.5, 0) circle (1.5pt);
    
    \draw (5.5, 0) -- (6.5, 0);
    \fill (6.5, 0) circle (1.5pt);
    \fill (6.5, 1) circle (1.5pt);

\end{tikzpicture}
    \caption{Possible ways to extend a spanning tree in $\mathscr{T}_{2n}$ to a forest in $\mathscr{F}_{2(n+1)}$}
    \label{fig:placeholder}
\end{figure}
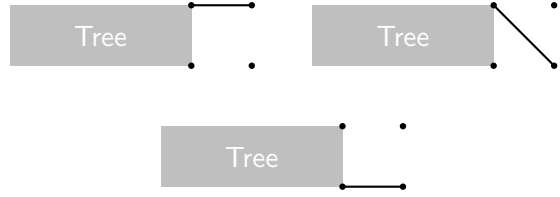

And there are two ways to extend a forest in $\mathscr{F}_{2n}$ to a forest in $\mathscr{F}_{2(n+1)}$:
\begin{figure}[H]
    \centering
        \begin{tikzpicture}[thick, every node/.style={font=\sffamily\large},xscale=0.8,yscale=0.8]

    \fill[lightgray] (0, 0) rectangle (3, 1);
    \node at (1.5, 0.5) {\textcolor{white}{Forest}};
    \fill (3, 1) circle (1.5pt);
    \fill (3, 0) circle (1.5pt);
    
    \draw (3, 0) -- (4, 0);
    \fill (4, 0) circle (1.5pt);
    \fill (4, 1) circle (1.5pt);
    \draw (3, 1) -- (4, 1);
    \fill[lightgray] (5, 0) rectangle (8, 1);
    \node at (6.5, 0.5) {\textcolor{white}{Forest}};
    
    \fill (8, 1) circle (1.5pt);
    \fill (8, 0) circle (1.5pt);
    \fill (9, 1) circle (1.5pt);
    
    \draw (8, 1) -- (9, 0);
    \fill (9, 0) circle (1.5pt);
    \draw (8, 0) -- (9, 0);

\end{tikzpicture}
    \caption{Possible ways to extend a forest in $\mathscr{F}_{2n}$ to a forest in $\mathscr{F}_{2(n+1)}$}
    \label{fig:placeholder}
\end{figure}
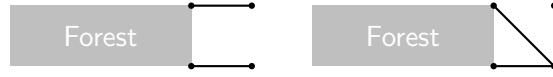

Thus, we get the recurrence 
\[F_{2(n+1)} = 3T_{2n} + 2F_{2n}. \] 

Our goal now is to combine the preceding recurrences to obtain a recurrence involving only $T_{2n}$. First, we let $v_{2n}$ denote the column vector
\[v_{2n} = 
\begin{bmatrix}
    T_{2n} \\
    F_{2n}
\end{bmatrix}.
\]
And if we define matrix A as
\[A = 
\begin{bmatrix}
    5 & 3 \\
    3 & 2
\end{bmatrix},
\]
we satisfy \[v_{2(n+1)}=Av_{2n}.\]
Since the characteristic polynomial for A is \[X_A(\lambda) = \lambda^2 - 7\lambda + 1,\] by applying the Cayley-Hamilton Theorem, we obtain
\[A^2 -7A +I = 0,\] which can be rewritten as
\[A^2 = 7A -I.\]
Multiplying $v_{2n}$ on both sides, we obtain
\[v_{2(n+2)} = 7v_{2(n+1)} - v_{2n}.\]
Hence:
\begin{equation}\label{eqn:recurrencerelationspanningtree} 
    \begin{aligned}
    T_{2(n+2)} = 7T_{2(n+1)} - T_{2n},\\
    \text{ and } F_{2(n+2)} = 7F_{2(n+1)} - F_{2n}.
    \end{aligned}
\end{equation}
From enumerating all possible spanning trees in $\mathscr{T}_2$ and $\mathscr{T}_4$ and all possible spanning forests in $\mathscr{F}_2$ and $\mathscr{F}_4$, we find the starting conditions: 

\begin{equation}\label{eqn:initialvalues1}
    \begin{aligned}
    T_2 = 1, \quad T_4 = 8,\\
    \quad \text{and} \quad
    F_2 = 1, \quad F_4 = 5.
    \end{aligned}
\end{equation}
\end{proof}

\subsection{Counting spanning trees of a triangular strip lattice with an odd number of vertices.}

We also consider triangular strip lattices with an odd number of vertices, since enumerating their spanning trees is a necessary prerequisite for counting balanced spanning trees in Section \ref{sec:countingbst}. Let $\Gamma_{2n-1}$ be obtained from $\Gamma_{2(n-1)}$ by adding a vertex to the right end of the lattice and connecting it to the two rightmost vertices (see Figure \ref{fig:Homtristrip11}). 

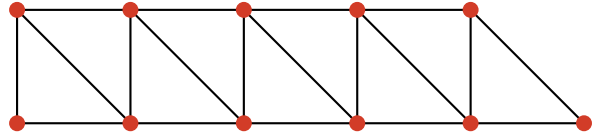
\begin{figure}[H]
    \centering
        \begin{tikzpicture}[scale=1.5, thick]
        
        \definecolor{dotcolor}{RGB}{210, 60, 40}

        \draw (0,1) -- (4,1); 
        \draw (0,0) -- (5,0); 

        \foreach \x in {0,1,2,3,4} {
            \draw (\x, 1) -- ++(1, -1); 
        }

        \foreach \x in {0,1,2,3,4} {
            \draw (\x,0) -- (\x,1); 
            \fill[dotcolor] (\x,1) circle (2pt);
            \fill[dotcolor] (\x,0) circle (2pt);
        }
        
        \fill[dotcolor] (5,0) circle (2pt);

    \end{tikzpicture}
    \caption{The graph of $\Gamma_{11},$ obtained by connecting a new vertex to the two rightmost vertices of $\Gamma_{10}.$}
    \label{fig:Homtristrip11}
\end{figure}

Let $T_{2n-1}$ denote the number of spanning trees of $\Gamma_{2n-1}$. The following theorem gives a recurrence for $T_{2n-1}$. 

\begin{theorem}[Number of spanning trees in $\Gamma_{2n - 1}$]\label{thm:sthomogenousodd}
The number of spanning trees in $\Gamma_{2n - 1}$, denoted by $T_{2n - 1}$ satisfies the recurrence relation: 
\[\ T_{2n+3} = 7T_{2n+1} - T_{2n - 1}.\]
with initial conditions $T_3 = 3$ and $T_5 = 21.$
\end{theorem}
\begin{proof} 
    \par 
    Since the orientation of the diagonal edge is determined by the structure of $\Gamma_{2k-1}$, we may assume without loss of generality that the diagonal joins the upper-left and lower-right vertices. The opposite orientation is symmetric and therefore yields the same number of spanning trees.
    \par
    Now, we follow the counting technique as the one presented in the last section with some minor changes. We also use the recurrence relation established in the last section to make our work easier. As before, we denote the rightmost vertices of $\Gamma_{2(n - 1)}$ by $\{u, v\}$. To get a spanning tree in $\Gamma_{2n - 1}$, we can either expand from a spanning tree in $\Gamma_{2(n - 1)}$ or expand from a forest in $\Gamma_{2(n - 1)}$ consisting of two trees each containing a distinct vertex in $\{u, v\}$. Observe from \Cref{fig:R_{2n - 1}} that there are only two ways to get a spanning tree in $\Gamma_{2n - 1}$ from a spanning tree in $\Gamma_{2(n - 1)}$ and only one way  from a forest in $\Gamma_{2(n - 1)}$.

    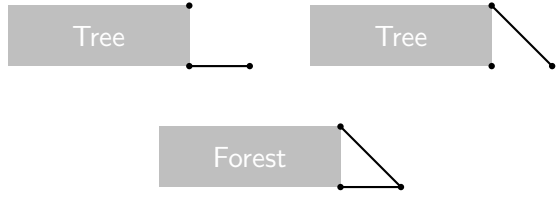
\begin{figure}[htbp]
        \centering
        \begin{tikzpicture}[thick, every node/.style={font=\sffamily\large},xscale=0.8,yscale=0.8]
    \fill[lightgray] (0, 2) rectangle (3, 3);
    \node at (1.5, 2.5) {\textcolor{white}{Tree}};

    \fill (3, 3) circle (1.5pt);
    \fill (3, 2) circle (1.5pt);
    
    \draw (3, 2) -- (4, 2);
    \fill (4, 2) circle (1.5pt);

    \fill[lightgray] (5, 2) rectangle (8, 3);
    \node at (6.5, 2.5) {\textcolor{white}{Tree}};
    
    \fill (8, 3) circle (1.5pt);
    \fill (8, 2) circle (1.5pt);
    
    \draw (8, 3) -- (9, 2);
    \fill (9, 2) circle (1.5pt);

    \fill[lightgray] (2.5, 0) rectangle (5.5, 1);
    \node at (4, 0.5) {\textcolor{white}{Forest}};
    
    \fill (5.5, 1) circle (1.5pt);
    \fill (5.5, 0) circle (1.5pt);
    
    \draw (5.5, 1) -- (6.5, 0);
    \draw (5.5, 0) -- (6.5, 0);
    \fill (6.5, 0) circle (1.5pt);

\end{tikzpicture}
        \caption{The possible ways to obtain a spanning tree in $\Gamma_{2n - 1}$.}
        \label{fig:R_{2n - 1}}
    \end{figure}
    \par
    So, from this, we obtain the following recurrence relation: 
    \[ T_{2n + 1} = 2T_{2n} + F_{2n}.\]
    
    Using the recurrence relation from (\ref{eqn:recurrencerelationspanningtree}) and making the appropriate substitutions, we get the following: 
    \begin{equation} \label{eqn:recurrencerelationspanningtreeodd} 
     T_{2n+3} = 7T_{2n+1} - T_{2n-1}
 \end{equation}
    Also, from (\ref{eqn:initialvalues1}), we can find the initial values $T_3$ and $T_5$.
    \begin{equation}\label{eqn:initialvalues2}
    T_3 = 2T_2 + F_2 = 2(1) + 1 = 3
    \quad \text{and} \quad
    T_5 = 2T_4 + F_4 = 2(8) + 5 = 21,
\end{equation} which completes our proof.
\end{proof}

\begin{remark}
Although Wang \cite{wang2023triangular} previously derived a recurrence for counting spanning trees in chains of triangles, a broader family of graphs that includes triangular strip lattices, we independently present a self-contained derivation specialized to triangular strip lattices. Although our techniques differ, the recurrences for triangular strip lattices in this section agree with the recurrence established by Wang \cite{wang2023triangular}:
\[
g_t = 3g_{t-1}-g_{t-2},
\]
where $g_t$ denotes the number of spanning trees in a triangular strip lattice with $t$ triangles. The proof of equivalence can be obtained by substituting $g_t=T_{t+2}.$
\end{remark}

\section{Counting balanced spanning trees}\label{sec:countingbst}
\subsection{Counting balanced spanning trees}
With a clear recurrence relation established for the total number of spanning trees, we now calculate the number of balanced spanning trees in a triangular strip lattice. Our derivation follows the general framework introduced by Gallagher and Tapp \cite{gal25}, adapted to the geometry of triangular strip lattices. 



To compute the total number of balanced spanning trees, our derivation proceeds in three main stages, as illustrated in \Cref{fig:overview}.

First, we find the number of certain ``valid'' loops of the \emph{dual graph} $\Gamma^*_{2n}$ (see Definition \ref{defn:dualgraph}) by counting the \emph{balanced partitions} of $\Gamma_{2n}.$ Each such loop separates the vertices of $\Gamma_{2n},$ as shown in red and purple in Figure \ref{fig:overview}. These loops will give rise to candidate balanced edges (bordered in orange). Next, we observe that replacing either component (in red or purple vertices) with any spanning tree on the same vertex set and connecting them with any balanced edge yields a valid balanced tree. Finally, we use this to generate all possible balanced spanning trees. 

For example, for the loop in \Cref{fig:overview} we have a total of $3T_6^2$ balanced spanning trees associated with it as we have $3$ possibilities for the balanced edge and $T_6$ possibilities for the spanning tree in either component.

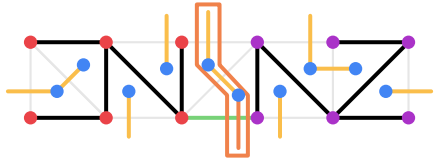
\begin{figure}[H]
    \centering
    \begin{tikzpicture}[scale=1, line join=round, line cap=round]
    
    \definecolor{gridgray}{RGB}{230, 230, 230}
    \definecolor{primalred}{RGB}{234, 67, 72}
    \definecolor{primalpurple}{RGB}{170, 50, 200}
    \definecolor{dualblue}{RGB}{66, 133, 244}
    \definecolor{dualyellow}{RGB}{250, 190, 75}
    \definecolor{centralgreen}{RGB}{120, 210, 120}
    \definecolor{centralorange}{RGB}{240, 130, 75}

    \draw[gridgray, thick] (0,0) -- (5,0);
    \draw[gridgray, thick] (0,1) -- (5,1);
    \foreach \x in {0,...,5} {
        \draw[gridgray, thick] (\x,0) -- (\x,1);
    }
    
    \draw[gridgray, thick] (0, 1) -- (1, 0);
    \draw[gridgray, thick] (0, 1) -- (1, 0);
    \draw[gridgray, thick] (2, 0) -- (3, 1);
    \draw[black, line width=1.5pt] (0,1) -- (1,1);
    \draw[black, line width=1.5pt] (0,0) -- (1,0);
    \draw[black, line width=1.5pt] (1,0) -- (1,1);
    \draw[black, line width=1.5pt] (1,1) -- (2,0);
    \draw[black, line width=1.5pt] (2,0) -- (2,1);

    \draw[black, line width=1.5pt] (3,1) -- (3,0);
    \draw[black, line width=1.5pt] (3,1) -- (4,0);
    \draw[black, line width=1.5pt] (4,0) -- (5,0);
    \draw[black, line width=1.5pt] (4,0) -- (5,1);
    \draw[black, line width=1.5pt] (4,1) -- (5,1);

    \draw[centralgreen, line width=1.5pt] (2,0) -- (3,0);

    \draw[centralorange, line width=1.5pt] 
        (2.2, 1.5) -- 
        (2.2, 0.7) -- 
        (2.6, 0.3) -- 
        (2.6, -0.5) -- 
        (2.875, -0.5) -- 
        (2.875, 0.3) -- 
        (2.5, 0.7) -- 
        (2.5, 1.5) -- cycle;

    \draw[dualyellow, line width=1.5pt] (-0.3, 0.35) -- (0.35, 0.35) -- (0.7, 0.7);
    
    \draw[dualyellow, line width=1.5pt] (1.3, 0.35) -- (1.3, -0.25);
    \draw[dualyellow, line width=1.5pt] (1.8, 0.65) -- (1.8, 1.35);

    \draw[dualyellow, line width=1.5pt] (2.35, 1.4) -- (2.35, 0.7) -- (2.75, 0.3);
    \draw[centralorange, line width=1.5pt] (2.75, 0.3) -- (2.75, -0.4);

    \draw[dualyellow, line width=1.5pt] (3.3, 0.35) -- (3.3, -0.25);
    \draw[dualyellow, line width=1.5pt] (3.7, 0.65) -- (4.3, 0.65);
    \draw[dualyellow, line width=1.5pt] (3.7, 0.65) -- (3.7, 1.35);
    \draw[dualyellow, line width=1.5pt] (4.7, 0.35) -- (5.3, 0.35);

    \fill[dualblue] (0.35, 0.35) circle (2.5pt);
    \fill[dualblue] (0.7, 0.7) circle (2.5pt);
    
    \fill[dualblue] (1.3, 0.35) circle (2.5pt);
    \fill[dualblue] (1.8, 0.65) circle (2.5pt);
    
    \fill[dualblue] (2.35, 0.7) circle (2.5pt);
    \fill[dualblue] (2.75, 0.3) circle (2.5pt);
    
    \fill[dualblue] (3.3, 0.35) circle (2.5pt);
    \fill[dualblue] (3.7, 0.65) circle (2.5pt);
    
    \fill[dualblue] (4.3, 0.65) circle (2.5pt);
    \fill[dualblue] (4.7, 0.35) circle (2.5pt);

    \foreach \x in {0,1,2} {
        \fill[primalred] (\x,0) circle (2.5pt);
        \fill[primalred] (\x,1) circle (2.5pt);
    }
    
    \foreach \x in {3,4,5} {
        \fill[primalpurple] (\x,0) circle (2.5pt);
        \fill[primalpurple] (\x,1) circle (2.5pt);
    }

\end{tikzpicture}
    \caption{Overview of the three-step framework for counting the total number of balanced spanning trees. }
    \label{fig:overview}
\end{figure}




We now provide a detailed derivation that encompasses all three stages. We begin by defining the dual graph of $\Gamma_{2n}.$

\begin{defn}[Dual graph]\label{defn:dualgraph}
    The dual graph of $\Gamma_{2n},$ denoted by $\Gamma_{2n}^{*},$ is the graph whose vertices correspond to the faces of 
    $\Gamma_{2n},$ including the unbounded exterior face, and whose edges correspond to the edges of $\Gamma_{2n}.$ Specifically, each edge $e \in E(\Gamma_{2n})$ has a corresponding dual edge $e^* \in E(\Gamma_{2n}^{*})$ connecting the two faces of $\Gamma_{2n}$ adjacent to $e$ (See Figure \ref{fig:dual_graph}).
\end{defn}


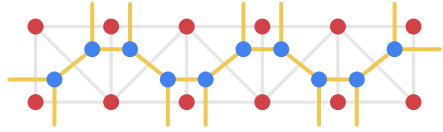
\begin{figure}[H]
    \centering
    \begin{tikzpicture}[scale=1, line join=round, line cap=round]
    
    \definecolor{gridgray}{RGB}{230, 230, 230}
    \definecolor{primalred}{RGB}{210, 65, 70}
    \definecolor{dualblue}{RGB}{65, 130, 240}
    \definecolor{dualyellow}{RGB}{240, 200, 80}

    \foreach \y in {0,1} {
        \draw[gridgray, line width=1.2pt] (0,\y) -- (5,\y);
    }
    
    \foreach \x in {0,...,5} {
        \draw[gridgray, line width=1.2pt] (\x,0) -- (\x,1);
    }
    
    \draw[gridgray, line width=1.2pt] (0,1) -- (1,0);
    \draw[gridgray, line width=1.2pt] (1,0) -- (2,1);
    \draw[gridgray, line width=1.2pt] (2,1) -- (3,0);
    \draw[gridgray, line width=1.2pt] (3,0) -- (4,1);
    \draw[gridgray, line width=1.2pt] (4,1) -- (5,0);

    \draw[dualyellow, line width=1.5pt] (-0.35, 0.3) -- (0.25, 0.3);
    
    \draw[dualyellow, line width=1.5pt] (4.75, 0.7) -- (5.35, 0.7);

    \draw[dualyellow, line width=1.5pt] 
        (0.25, 0.3) -- (0.75, 0.7) -- 
        (1.25, 0.7) -- (1.75, 0.3) -- 
        (2.25, 0.3) -- (2.75, 0.7) -- 
        (3.25, 0.7) -- (3.75, 0.3) -- 
        (4.25, 0.3) -- (4.75, 0.7);

    \foreach \x in {0.25, 1.75, 2.25, 3.75, 4.25} {
        \draw[dualyellow, line width=1.5pt] (\x, 0.3) -- (\x, -0.3);
    }
    
    \foreach \x in {0.75, 1.25, 2.75, 3.25, 4.75} {
        \draw[dualyellow, line width=1.5pt] (\x, 0.7) -- (\x, 1.3);
    }

    \foreach \x/\y in {
        0.25/0.3, 0.75/0.7,
        1.25/0.7, 1.75/0.3,
        2.25/0.3, 2.75/0.7,
        3.25/0.7, 3.75/0.3,
        4.25/0.3, 4.75/0.7
    } {
        \fill[dualblue] (\x, \y) circle (3pt);
    }

    \foreach \x in {0,...,5} {
        \foreach \y in {0,1} {
            \fill[primalred] (\x, \y) circle (3pt);
        }
    }

\end{tikzpicture}
    \caption{Dual graph of a triangular strip lattice. The vertices of $\Gamma_{2n}$ are shown in red, the edges are grey, while the vertices of $\Gamma_{2n}^*$ are shown in blue and the edges are in yellow. The vertex $v_\infty^*$ is omitted for clarity.}
    \label{fig:dual_graph}
\end{figure}
There is a one-to-one correspondence between the spanning trees of $\Gamma_{2n}$ and the spanning trees of its dual $\Gamma_{2n}^*.$ For any spanning tree $\mathcal{T} \in \mathscr{T}_{2n}$, its corresponding spanning tree $\mathcal{T}^*    \in \mathscr{T}^*_{2n}$ consists precisely of edges $e^*\in \Gamma_{2n}^*$ that cross edges $e \in \Gamma_{2n} \setminus \mathcal{T}.$ See \Cref{fig:spanningtreeexam}.

\begin{figure}[htbp]
     \centering
     \begin{tikzpicture}[scale=1, line join=round, line cap=round]
    
    \definecolor{gridgray}{RGB}{230, 230, 230}
    \definecolor{primalred}{RGB}{234, 67, 72}
    \definecolor{dualblue}{RGB}{66, 133, 244}
    \definecolor{dualyellow}{RGB}{250, 190, 75}

    \foreach \y in {0,1} {
        \draw[gridgray, line width=1.2pt] (0,\y) -- (4,\y);
    }
    
    \foreach \x in {0,...,4} {
        \draw[gridgray, line width=1.2pt] (\x,0) -- (\x,1);
    }
    
    \draw[gridgray, line width=1.2pt] (0,1) -- (1,0);
    \draw[gridgray, line width=1.2pt] (3,1) -- (4,0);

    \draw[dualyellow, line width=1.5pt] (-0.3, 0.35) -- (0.35, 0.35) -- (0.7, 0.7);
    
    \draw[dualyellow, line width=1.5pt] (1.25, 0.4) -- (1.25, -0.25);
    \draw[dualyellow, line width=1.5pt] (1.7, 0.7) -- (1.7, 1.25);
    
    \draw[dualyellow, line width=1.5pt] (2.25, 0.4) -- (2.25, -0.25);
    
    \draw[dualyellow, line width=1.5pt] (2.7, 0.7) -- (3.35, 0.35) -- (3.75, 0.7);
    
    \draw[dualyellow, line width=1.5pt] (3.75, 0.7) -- (3.75, 1.25);

    \draw[black, line width=1.5pt] (0,1) -- (1,1);
    \draw[black, line width=1.5pt] (0,0) -- (1,0);
    \draw[black, line width=1.5pt] (1,0) -- (1,1);
    
    \draw[black, line width=1.5pt] (1,1) -- (2,0);
    \draw[black, line width=1.5pt] (2,0) -- (2,1);
    
    \draw[black, line width=1.5pt] (2,1) -- (3,1);
    \draw[black, line width=1.5pt] (2,1) -- (3,0);
    
    \draw[black, line width=1.5pt] (3,0) -- (4,0);
    \draw[black, line width=1.5pt] (4,0) -- (4,1);

    \fill[dualblue] (0.35, 0.35) circle (2.5pt);
    \fill[dualblue] (0.7, 0.7) circle (2.5pt);
    
    \fill[dualblue] (1.25, 0.4) circle (2.5pt);
    \fill[dualblue] (1.7, 0.7) circle (2.5pt);
    
    \fill[dualblue] (2.25, 0.4) circle (2.5pt);
    \fill[dualblue] (2.7, 0.7) circle (2.5pt);
    
    \fill[dualblue] (3.35, 0.35) circle (2.5pt);
    \fill[dualblue] (3.75, 0.7) circle (2.5pt);

    \foreach \x in {0,...,4} {
        \foreach \y in {0,1} {
            \fill[primalred] (\x, \y) circle (2.5pt);
        }
    }

\end{tikzpicture}
     \caption{A spanning tree of $\Gamma_{10}$ (in red and black) and the corresponding spanning tree in $\Gamma_{10}^*$ (in yellow and blue) for the dual graph ($v_{\infty}^*$ not shown explicitly).}
     \label{fig:spanningtreeexam}
\end{figure}
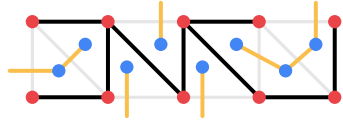

Under this formulation, removing a single edge $e \in E(\mathcal{T})$ to form a forest $\mathcal{T} \setminus \{e\}$ is equivalent to adding its corresponding edge $e^*$ to $\mathcal{T}^*$, which creates a unique valid loop $\gamma$ in $\mathcal{T}^* \cup \{e^*\}$. (See Figure \ref{fig:looptodiv}.)

    \par
     Let $\mathcal{T} \in \mathscr{T}_{2n}$ be a balanced spanning tree, and let $e \in E(\mathcal{T})$ be its balanced edge whose removal partitions $V(\Gamma_{2n})$ into two components of equal size. 
    The removal of $e$ induces the addition of $e^* \in \Gamma_{2n}^*$ and creates a unique cycle $\gamma$ in $\mathcal{T}^* \cup\{e^*\}.$ This cycle passes through the vertex $v_\infty^{\ast},$ and we call it a \emph{valid loop}.
 \par 
    Now, we claim the following: 

\begin{claim}[Number of valid loops]\label{claim:loops} $\Gamma_{2n}$ has n valid loops.
\end{claim}




\begin{proof} Observe that, when a valid loop is formed by the removal of the balanced edge, the vertices of $\Gamma_{2n}$ are partitioned into two sets, each having $n$ vertices (see \Cref{fig:looptodiv} for the case where $n = 5$). We define a \emph{balanced partition} as a partition of the vertices of $\Gamma_{2n}$ into two disjoint, sets having $n$ vertices. For convenience, denote by $L$ and $P$ the set of valid loops and balanced partitions respectively. There is a bijective correspondence between valid loops and balanced partitions. Therefore, $L$ and $P$ have the same cardinality. We now count the number of balanced partitions.

\begin{figure}[H]
    \centering
    \begin{tikzpicture}[scale=1, line join=round, line cap=round]
    
    \definecolor{gridgray}{RGB}{230, 230, 230}
    \definecolor{primalred}{RGB}{234, 67, 72}
    \definecolor{primalpurple}{RGB}{170, 50, 200}
    \definecolor{centralgreen}{RGB}{100, 200, 100}
    \definecolor{dualblue}{RGB}{66, 133, 244}
    \definecolor{dualyellow}{RGB}{250, 190, 75}

    \begin{scope}[yshift=1.8cm]
        
        \foreach \y in {0,1} {
            \draw[gridgray, line width=1.2pt] (0,\y) -- (4,\y);
        }
        \foreach \x in {0,...,4} {
            \draw[gridgray, line width=1.2pt] (\x,0) -- (\x,1);
        }
        \foreach \x in {0,...,3} {
            \draw[gridgray, line width=1.2pt] (\x,1) -- (\x+1,0);
        }

        \draw[centralgreen, line width=1.5pt] (2,0) -- (2,1);

        \draw[black, line width=1.5pt] (0,1) -- (1,1);
        \draw[black, line width=1.5pt] (0,0) -- (1,0);
        \draw[black, line width=1.5pt] (1,0) -- (1,1);
        \draw[black, line width=1.5pt] (1,1) -- (2,0);
        
        \draw[black, line width=1.5pt] (2,1) -- (3,1);
        \draw[black, line width=1.5pt] (2,1) -- (3,0);
        \draw[black, line width=1.5pt] (3,0) -- (4,0);
        \draw[black, line width=1.5pt] (4,0) -- (4,1);

        \draw[dualyellow, line width=1.5pt] (-0.35, 0.3) -- (0.35, 0.3) -- (0.7, 0.7);
        
        \draw[dualyellow, line width=1.5pt] (1.25, 0.3) -- (1.25, -0.3);
        \draw[dualyellow, line width=1.5pt] (1.7, 0.7) -- (1.7, 1.3);
        
        \draw[dualyellow, line width=1.5pt] (2.25, 0.3) -- (2.25, -0.3);
        
        \draw[dualyellow, line width=1.5pt] (2.7, 0.7) -- (3.25, 0.3) -- (3.75, 0.7);
        
        \draw[dualyellow, line width=1.5pt] (3.75, 0.7) -- (3.75, 1.3);

        \fill[dualblue] (0.35, 0.3) circle (2.5pt);
        \fill[dualblue] (0.7, 0.7) circle (2.5pt);
        \fill[dualblue] (1.25, 0.3) circle (2.5pt);
        \fill[dualblue] (1.7, 0.7) circle (2.5pt);
        \fill[dualblue] (2.25, 0.3) circle (2.5pt);
        \fill[dualblue] (2.7, 0.7) circle (2.5pt);
        \fill[dualblue] (3.25, 0.3) circle (2.5pt);
        \fill[dualblue] (3.75, 0.7) circle (2.5pt);

        \foreach \p in {(0,0), (1,0), (2,0), (0,1), (1,1)} {
            \fill[primalred] \p circle (2.5pt);
        }
        \foreach \p in {(2,1), (3,0), (4,0), (3,1), (4,1)} {
            \fill[primalpurple] \p circle (2.5pt);
        }
    \end{scope}

    \begin{scope}[yshift=0cm]
        
        \foreach \y in {0,1} {
            \draw[gridgray, line width=1.2pt] (0,\y) -- (4,\y);
        }
        \foreach \x in {0,...,4} {
            \draw[gridgray, line width=1.2pt] (\x,0) -- (\x,1);
        }
        \foreach \x in {0,...,3} {
            \draw[gridgray, line width=1.2pt] (\x,1) -- (\x+1,0);
        }


        \draw[black, line width=1.5pt] (0,1) -- (1,1);
        \draw[black, line width=1.5pt] (0,0) -- (1,0);
        \draw[black, line width=1.5pt] (1,0) -- (1,1);
        \draw[black, line width=1.5pt] (1,1) -- (2,0);
        
        \draw[black, line width=1.5pt] (2,1) -- (3,1);
        \draw[black, line width=1.5pt] (2,1) -- (3,0);
        \draw[black, line width=1.5pt] (3,0) -- (4,0);
        \draw[black, line width=1.5pt] (4,0) -- (4,1);

        \draw[dualyellow, line width=1.5pt] (-0.35, 0.3) -- (0.35, 0.3) -- (0.7, 0.7);
        
        \draw[dualyellow, line width=1.5pt] (1.25, 0.3) -- (1.25, -0.3);
        \draw[dualyellow, line width=1.5pt] (1.7, 0.7) -- (1.7, 1.3);
        
        \draw[dualyellow, line width=1.5pt] (1.7, 0.7) -- (2.25, 0.3);
        
        \draw[dualyellow, line width=1.5pt] (2.25, 0.3) -- (2.25, -0.3);
        
        \draw[dualyellow, line width=1.5pt] (2.7, 0.7) -- (3.25, 0.3) -- (3.75, 0.7);
        
        \draw[dualyellow, line width=1.5pt] (3.75, 0.7) -- (3.75, 1.3);

        \fill[dualblue] (0.35, 0.3) circle (2.5pt);
        \fill[dualblue] (0.7, 0.7) circle (2.5pt);
        \fill[dualblue] (1.25, 0.3) circle (2.5pt);
        \fill[dualblue] (1.7, 0.7) circle (2.5pt);
        \fill[dualblue] (2.25, 0.3) circle (2.5pt);
        \fill[dualblue] (2.7, 0.7) circle (2.5pt);
        \fill[dualblue] (3.25, 0.3) circle (2.5pt);
        \fill[dualblue] (3.75, 0.7) circle (2.5pt);

        \foreach \p in {(0,0), (1,0), (2,0), (0,1), (1,1)} {
            \fill[primalred] \p circle (2.5pt);
        }
        \foreach \p in {(2,1), (3,0), (4,0), (3,1), (4,1)} {
            \fill[primalpurple] \p circle (2.5pt);
        }
    \end{scope}

\end{tikzpicture}
    \caption{The removal of the balanced edge in $\Gamma_{10}$ (green edge at top) creates a valid loop in the dual $\Gamma^{\ast}_{10}$ that gives a balanced partition of the vertex set of $\Gamma_{10}$.}
    \label{fig:looptodiv}
\end{figure}
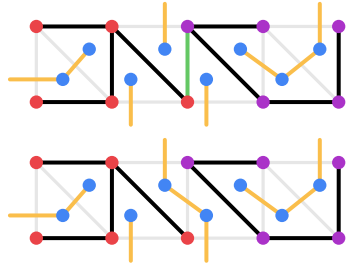

        Let the two sets of vertices we get in a balanced partition be $A$ and $B$, colored by red and purple respectively in \Cref{fig:looptodiv} and \Cref{fig:balancedpartitionex}. We can get a balanced partition by taking in $A$, the leftmost $k$ vertices from the top row and leftmost $n - k$ vertices from the bottom row where $k \in \{ 1, 2, 3, \dots, n - 1, n\}$ and including the rest of the vertices of $\Gamma_{2n}$ in $B$.
        
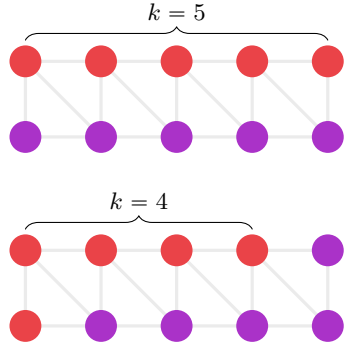
\begin{figure}[H]
    \centering
    \begin{tikzpicture}[scale=1, line join=round, line cap=round]
    
    \definecolor{gridgray}{RGB}{235, 235, 235}
    \definecolor{primalred}{RGB}{234, 67, 72}
    \definecolor{primalpurple}{RGB}{170, 50, 200}

    \tikzset{
        reddot/.style={circle, fill=primalred, minimum size=12pt, inner sep=0},
        purpledot/.style={circle, fill=primalpurple, minimum size=12pt, inner sep=0},
        redlabel/.style={circle, fill=primalred, text=white, font=\sffamily\bfseries\scriptsize, minimum size=12pt, inner sep=0}
    }
 \usetikzlibrary{decorations.pathreplacing}

\draw[decorate, decoration={brace, amplitude=5pt, raise=8pt}] 
  (0, 1) -- (3, 1) 
  node[pos=0.5, above=12pt, font=\small] {$k = 4$};

    \begin{scope}[yshift=2.5cm]
        \draw[decorate, decoration={brace, amplitude=5pt, raise=8pt}] 
  (0, 1) -- (4, 1) 
  node[pos=0.5, above=12pt, font=\small] {$k = 5$};

        \foreach \y in {0,1} {
            \draw[gridgray, line width=1.2pt] (0,\y) -- (4,\y);
        }
        \foreach \x in {0,...,4} {
            \draw[gridgray, line width=1.2pt] (\x,0) -- (\x,1);
        }
        \foreach \x in {0,...,3} {
            \draw[gridgray, line width=1.2pt] (\x,1) -- (\x+1,0);
        }

        \foreach \x in {0,...,4} {
            \node[purpledot] at (\x, 0) {};
            \node[redlabel] at (\x, 1) {};
        }
    \end{scope}


    \begin{scope}[yshift=0cm]
        \foreach \y in {0,1} {
            \draw[gridgray, line width=1.2pt] (0,\y) -- (4,\y);
        }
        \foreach \x in {0,...,4} {
            \draw[gridgray, line width=1.2pt] (\x,0) -- (\x,1);
        }
        \foreach \x in {0,...,3} {
            \draw[gridgray, line width=1.2pt] (\x,1) -- (\x+1,0);
        }

        \node[reddot] at (0, 0) {};
        \foreach \x in {1,...,4} {
            \node[purpledot] at (\x, 0) {};
        }

        \foreach \x in {0,...,3} {
            \node[redlabel] at (\x, 1) {};
        }
        \node[purpledot] at (4, 1) {};
    \end{scope}

\end{tikzpicture}
    \caption{Two balanced partitions corresponding to $k$ = 5 and $k$ = 4 respectively.}
    \label{fig:balancedpartitionex}
\end{figure}

    From this, we can see that there are a total of $n$ balanced partitions. Since $L$ and $P$ have the same cardinality, and $P$ has $n$ elements, $L$ must also have $n$ elements. 

    \par
    This proves our claim.
    \end{proof}

Next, following the indexing convention of Gallagher and Tapp \cite{gal25}, we label the $n$ partitions as shown in Figure \ref{fig:partitions_ex}. Define $m=\lfloor n/2\rfloor$. The indexing is chosen so that there are $m$ indexed partitions, and reflection across the vertical centerline pairs most partitions. Specifically, the non-symmetric partitions $1,\dots,m-1$ for even $n$ and $0,\dots,m-1$ for odd $n$ are paired with their reflections. For convenience, indexed partitions are precisely those for which the left side of the partition contains at least as many bottom-row vertices as top-row vertices.


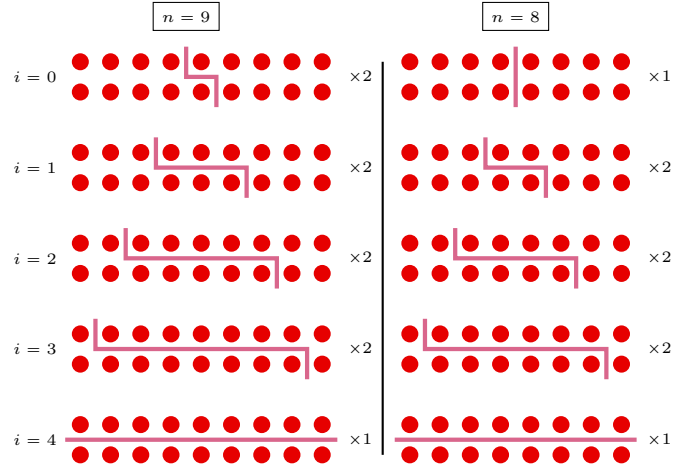
\begin{figure}[H]
    \centering
    \begin{tikzpicture}[xscale=0.4,yscale=0.4]
        \node [font=\fontsize{6}{14}, draw] at (3.5,14.5){$n=9$};

        n=9
        i=4
        \node[font=\fontsize{6}{1}] at (-1.5,0.5) {$i=4$};
        \node[font=\fontsize{6}{1}] at (9.25,0.5) {$\times1$};
        \fill[red!90!black] (0,0) circle(8pt);
        \fill[red!90!black] (1,0) circle(8pt);
        \fill[red!90!black] (2,0) circle(8pt);
        \fill[red!90!black] (3,0) circle(8pt);
        \fill[red!90!black] (4,0) circle(8pt);
        \fill[red!90!black] (5,0) circle(8pt);
        \fill[red!90!black] (6,0) circle(8pt);
        \fill[red!90!black] (7,0) circle(8pt);
        \fill[red!90!black] (8,0) circle(8pt);
        \fill[red!90!black] (0,1) circle(8pt);
        \fill[red!90!black] (1,1) circle(8pt);
        \fill[red!90!black] (2,1) circle(8pt);
        \fill[red!90!black] (3,1) circle(8pt);
        \fill[red!90!black] (4,1) circle(8pt);
        \fill[red!90!black] (5,1) circle(8pt);
        \fill[red!90!black] (6,1) circle(8pt);
        \fill[red!90!black] (7,1) circle(8pt);
        \fill[red!90!black] (8,1) circle(8pt);
        \draw[ultra thick, purple!60] (3.5,13.5)--(3.5,12.5)--(4.5,12.5)--(4.5,11.5);
        
        i=3
        \node[font=\fontsize{6}{1}] at (-1.5,3.5) {$i=3$};
        \node[font=\fontsize{6}{1}] at (9.25,3.5) {$\times2$};
        \fill[red!90!black] (0,3) circle(8pt);
        \fill[red!90!black] (1,3) circle(8pt);
        \fill[red!90!black] (2,3) circle(8pt);
        \fill[red!90!black] (3,3) circle(8pt);
        \fill[red!90!black] (4,3) circle(8pt);
        \fill[red!90!black] (5,3) circle(8pt);
        \fill[red!90!black] (6,3) circle(8pt);
        \fill[red!90!black] (7,3) circle(8pt);
        \fill[red!90!black] (8,3) circle(8pt);
        \fill[red!90!black] (0,4) circle(8pt);
        \fill[red!90!black] (1,4) circle(8pt);
        \fill[red!90!black] (2,4) circle(8pt);
        \fill[red!90!black] (3,4) circle(8pt);
        \fill[red!90!black] (4,4) circle(8pt);
        \fill[red!90!black] (5,4) circle(8pt);
        \fill[red!90!black] (6,4) circle(8pt);
        \fill[red!90!black] (7,4) circle(8pt);
        \fill[red!90!black] (8,4) circle(8pt);
        \draw[ultra thick, purple!60] (2.5,10.5)--(2.5,9.5)--(5.5,9.5)--(5.5,8.5);

        i=2
        \node[font=\fontsize{6}{1}] at (-1.5,6.5) {$i=2$};
        \node[font=\fontsize{6}{1}] at (9.25,6.5) {$\times2$};
        \fill[red!90!black] (0,6) circle(8pt);
        \fill[red!90!black] (1,6) circle(8pt);
        \fill[red!90!black] (2,6) circle(8pt);
        \fill[red!90!black] (3,6) circle(8pt);
        \fill[red!90!black] (4,6) circle(8pt);
        \fill[red!90!black] (5,6) circle(8pt);
        \fill[red!90!black] (6,6) circle(8pt);
        \fill[red!90!black] (7,6) circle(8pt);
        \fill[red!90!black] (8,6) circle(8pt);
        \fill[red!90!black] (0,7) circle(8pt);
        \fill[red!90!black] (1,7) circle(8pt);
        \fill[red!90!black] (2,7) circle(8pt);
        \fill[red!90!black] (3,7) circle(8pt);
        \fill[red!90!black] (4,7) circle(8pt);
        \fill[red!90!black] (5,7) circle(8pt);
        \fill[red!90!black] (6,7) circle(8pt);
        \fill[red!90!black] (7,7) circle(8pt);
        \fill[red!90!black] (8,7) circle(8pt);
        \draw[ultra thick, purple!60] (1.5,7.5)--(1.5,6.5)--(6.5,6.5)--(6.5,5.5);

        i=1
        \node[font=\fontsize{6}{1}] at (-1.5,9.5) {$i=1$};
        \node[font=\fontsize{6}{1}] at (9.25,9.5) {$\times2$};
        \fill[red!90!black] (0,9) circle(8pt);
        \fill[red!90!black] (1,9) circle(8pt);
        \fill[red!90!black] (2,9) circle(8pt);
        \fill[red!90!black] (3,9) circle(8pt);
        \fill[red!90!black] (4,9) circle(8pt);
        \fill[red!90!black] (5,9) circle(8pt);
        \fill[red!90!black] (6,9) circle(8pt);
        \fill[red!90!black] (7,9) circle(8pt);
        \fill[red!90!black] (8,9) circle(8pt);
        \fill[red!90!black] (0,10) circle(8pt);
        \fill[red!90!black] (1,10) circle(8pt);
        \fill[red!90!black] (2,10) circle(8pt);
        \fill[red!90!black] (3,10) circle(8pt);
        \fill[red!90!black] (4,10) circle(8pt);
        \fill[red!90!black] (5,10) circle(8pt);
        \fill[red!90!black] (6,10) circle(8pt);
        \fill[red!90!black] (7,10) circle(8pt);
        \fill[red!90!black] (8,10) circle(8pt);
        \draw[ultra thick, purple!60] (0.5,4.5)--(0.5,3.5)--(7.5,3.5)--(7.5,2.5);

        i=0
        \node[font=\fontsize{6}{1}] at (-1.5,12.5) {$i=0$};
        \node[font=\fontsize{6}{1}] at (9.25,12.5) {$\times2$};
        \fill[red!90!black] (0,12) circle(8pt);
        \fill[red!90!black] (1,12) circle(8pt);
        \fill[red!90!black] (2,12) circle(8pt);
        \fill[red!90!black] (3,12) circle(8pt);
        \fill[red!90!black] (4,12) circle(8pt);
        \fill[red!90!black] (5,12) circle(8pt);
        \fill[red!90!black] (6,12) circle(8pt);
        \fill[red!90!black] (7,12) circle(8pt);
        \fill[red!90!black] (8,12) circle(8pt);
        \fill[red!90!black] (0,13) circle(8pt);
        \fill[red!90!black] (1,13) circle(8pt);
        \fill[red!90!black] (2,13) circle(8pt);
        \fill[red!90!black] (3,13) circle(8pt);
        \fill[red!90!black] (4,13) circle(8pt);
        \fill[red!90!black] (5,13) circle(8pt);
        \fill[red!90!black] (6,13) circle(8pt);
        \fill[red!90!black] (7,13) circle(8pt);
        \fill[red!90!black] (8,13) circle(8pt);
        \draw[ultra thick, purple!60] (-.5,.5)--(8.5,0.5);

        dividing line
        \draw[thick] (10,13)--(10,0);
    \end{tikzpicture}
    \begin{tikzpicture}[xscale=0.4,yscale=0.4]
        \node [font=\fontsize{6}{14}, draw] at (3.5,14.5){$n=8$};
        n=9
        i=4
        \node[font=\fontsize{6}{1}] at (8.25,0.5) {$\times1$};
        \fill[red!90!black] (0,0) circle(8pt);
        \fill[red!90!black] (1,0) circle(8pt);
        \fill[red!90!black] (2,0) circle(8pt);
        \fill[red!90!black] (3,0) circle(8pt);
        \fill[red!90!black] (4,0) circle(8pt);
        \fill[red!90!black] (5,0) circle(8pt);
        \fill[red!90!black] (6,0) circle(8pt);
        \fill[red!90!black] (7,0) circle(8pt);
        \fill[red!90!black] (0,1) circle(8pt);
        \fill[red!90!black] (1,1) circle(8pt);
        \fill[red!90!black] (2,1) circle(8pt);
        \fill[red!90!black] (3,1) circle(8pt);
        \fill[red!90!black] (4,1) circle(8pt);
        \fill[red!90!black] (5,1) circle(8pt);
        \fill[red!90!black] (6,1) circle(8pt);
        \fill[red!90!black] (7,1) circle(8pt);
        \draw[ultra thick, purple!60] (3.5,13.5)--(3.5,11.5);
        
        i=3
        \node[font=\fontsize{6}{1}] at (8.25,3.5) {$\times2$};
        \fill[red!90!black] (0,3) circle(8pt);
        \fill[red!90!black] (1,3) circle(8pt);
        \fill[red!90!black] (2,3) circle(8pt);
        \fill[red!90!black] (3,3) circle(8pt);
        \fill[red!90!black] (4,3) circle(8pt);
        \fill[red!90!black] (5,3) circle(8pt);
        \fill[red!90!black] (6,3) circle(8pt);
        \fill[red!90!black] (7,3) circle(8pt);
        \fill[red!90!black] (0,4) circle(8pt);
        \fill[red!90!black] (1,4) circle(8pt);
        \fill[red!90!black] (2,4) circle(8pt);
        \fill[red!90!black] (3,4) circle(8pt);
        \fill[red!90!black] (4,4) circle(8pt);
        \fill[red!90!black] (5,4) circle(8pt);
        \fill[red!90!black] (6,4) circle(8pt);
        \fill[red!90!black] (7,4) circle(8pt);
        \draw[ultra thick, purple!60] (2.5,10.5)--(2.5,9.5)--(4.5,9.5)--(4.5,8.5);

        i=2
        \node[font=\fontsize{6}{1}] at (8.25,6.5) {$\times2$};
        \fill[red!90!black] (0,6) circle(8pt);
        \fill[red!90!black] (1,6) circle(8pt);
        \fill[red!90!black] (2,6) circle(8pt);
        \fill[red!90!black] (3,6) circle(8pt);
        \fill[red!90!black] (4,6) circle(8pt);
        \fill[red!90!black] (5,6) circle(8pt);
        \fill[red!90!black] (6,6) circle(8pt);
        \fill[red!90!black] (7,6) circle(8pt);
        \fill[red!90!black] (0,7) circle(8pt);
        \fill[red!90!black] (1,7) circle(8pt);
        \fill[red!90!black] (2,7) circle(8pt);
        \fill[red!90!black] (3,7) circle(8pt);
        \fill[red!90!black] (4,7) circle(8pt);
        \fill[red!90!black] (5,7) circle(8pt);
        \fill[red!90!black] (6,7) circle(8pt);
        \fill[red!90!black] (7,7) circle(8pt);
        \draw[ultra thick, purple!60] (1.5,7.5)--(1.5,6.5)--(5.5,6.5)--(5.5,5.5);
        
        i=1
        \node[font=\fontsize{6}{1}] at (8.25,9.5) {$\times2$};
        \fill[red!90!black] (0,9) circle(8pt);
        \fill[red!90!black] (1,9) circle(8pt);
        \fill[red!90!black] (2,9) circle(8pt);
        \fill[red!90!black] (3,9) circle(8pt);
        \fill[red!90!black] (4,9) circle(8pt);
        \fill[red!90!black] (5,9) circle(8pt);
        \fill[red!90!black] (6,9) circle(8pt);
        \fill[red!90!black] (7,9) circle(8pt);
        \fill[red!90!black] (0,10) circle(8pt);
        \fill[red!90!black] (1,10) circle(8pt);
        \fill[red!90!black] (2,10) circle(8pt);
        \fill[red!90!black] (3,10) circle(8pt);
        \fill[red!90!black] (4,10) circle(8pt);
        \fill[red!90!black] (5,10) circle(8pt);
        \fill[red!90!black] (6,10) circle(8pt);
        \fill[red!90!black] (7,10) circle(8pt);
        \draw[ultra thick, purple!60] (0.5,4.5)--(0.5,3.5)--(6.5,3.5)--(6.5,2.5);

        i=0
        \node[font=\fontsize{6}{1}] at (8.25,12.5) {$\times1$};
        \fill[red!90!black] (0,12) circle(8pt);
        \fill[red!90!black] (1,12) circle(8pt);
        \fill[red!90!black] (2,12) circle(8pt);
        \fill[red!90!black] (3,12) circle(8pt);
        \fill[red!90!black] (4,12) circle(8pt);
        \fill[red!90!black] (5,12) circle(8pt);
        \fill[red!90!black] (6,12) circle(8pt);
        \fill[red!90!black] (7,12) circle(8pt);
        \fill[red!90!black] (0,13) circle(8pt);
        \fill[red!90!black] (1,13) circle(8pt);
        \fill[red!90!black] (2,13) circle(8pt);
        \fill[red!90!black] (3,13) circle(8pt);
        \fill[red!90!black] (4,13) circle(8pt);
        \fill[red!90!black] (5,13) circle(8pt);
        \fill[red!90!black] (6,13) circle(8pt);
        \fill[red!90!black] (7,13) circle(8pt);
        \draw[ultra thick, purple!60] (-.5,.5)--(7.5,.5);
    \end{tikzpicture}
    \caption{The $n$ possible ways to partition the vertices of $\Gamma_{2n}$ equally for $n = 9$ and $n = 8$. Inspired by \cite{gal25}.}
    \label{fig:partitions_ex}
\end{figure}

By Claim \ref{claim:loops}, each partition corresponds to a valid loop in $\Gamma_{2n}^*$. We now use $\gamma_i$ to denote the loop that corresponds to the indexed partition $i$, and $\mu_i$ to denote the loop that corresponds to the reflection of partition $i$. Note that since diagonals of faces break the symmetry that exists in a grid graph, $\gamma_i$ does not necessarily have the same length as $\mu_i.$ Therefore, it is important to distinguish them.

Now, we will use Figure \ref{fig:counting_balanced_trees} to illustrate our strategy of counting the number of balanced spanning trees associated with each valid loop. Each valid loop divides the vertex set of the graph into two equal sets. A balanced spanning tree can be obtained by first forming a spanning tree on either side of the loop, then connecting the two trees by replacing one edge of the valid loop with the edge it intersects in the graph $\Gamma_{2n}$. The number of spanning trees on either side of the loop equals the number of spanning trees in \emph{end blocks} ({highlighted} in yellow) of their respective sides. Therefore, if we denote the number of spanning trees in the left and right end blocks by $T_L$ and $T_R$ respectively, then each valid loop is associated with length$(\gamma_i)T_L T_R$ balanced spanning trees. Note that $\gamma_i$ can be interchanged with $\mu_i$ here.
\begin{figure}[H]
    \begin{tikzpicture}[thick,xscale=0.8,yscale=0.8]
        \fill[yellow!30] (0,0) rectangle (2,1);
        \fill[yellow!30] (6,1)--(9,1)--(9,0)--(7,0);
        \draw grid(9,1);
        \draw (0,1)--(1,0)--(2,1) (2,0)--(3,1) (3,0)--(4,1)--(5,0)--(6,1)--(7,0) (7,1)--(8,0) (8,1)--(9,0);
        \draw[red, very thick] (2.25,1.25)--(2.25,0.75)--(2.75,0.25)--(3.25,0.75)--(3.75,0.25)--(4.25,0.25)--(4.75,0.75)--(5.25,0.75)--(5.75,0.25)--(6.25,0.25)--(6.25,-0.25);
        
        \fill (2.25,0.75) circle (1.1pt);
        \fill (2.75,0.25) circle (1.1pt);
        \fill (3.25,0.75) circle (1.1pt);
        \fill (3.75,0.25) circle (1.1pt);
        \fill (4.25,0.25) circle (1.1pt);
        \fill (4.75,0.75) circle (1.1pt);
        \fill (5.25,0.75) circle (1.1pt);
        \fill (5.75,0.25) circle (1.1pt);
        \fill (6.25,0.25) circle (1.1pt);

        nodes
            \fill (0,0) circle (1.5pt);
            \fill (1,0) circle (1.5pt);
            \fill (2,0) circle (1.5pt);
            \fill (3,0) circle (1.5pt);
            \fill (4,0) circle (1.5pt);
            \fill (5,0) circle (1.5pt);
            \fill (6,0) circle (1.5pt);
            \fill (7,0) circle (1.5pt);
            \fill (8,0) circle (1.5pt);
            \fill (9,0) circle (1.5pt);
            \fill (0,1) circle (1.5pt);
            \fill (1,1) circle (1.5pt);
            \fill (2,1) circle (1.5pt);
            \fill (3,1) circle (1.5pt);
            \fill (4,1) circle (1.5pt);
            \fill (5,1) circle (1.5pt);
            \fill (6,1) circle (1.5pt);
            \fill (7,1) circle (1.5pt);
            \fill (8,1) circle (1.5pt);
            \fill (9,1) circle (1.5pt);
    \end{tikzpicture}
    \caption{Example of a valid loop}
    \label{fig:counting_balanced_trees}
\end{figure}
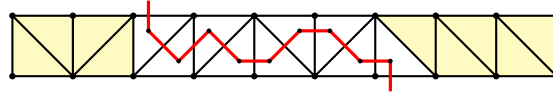

Observe that both the length of the loop and the sizes of the end blocks depend only on the types of diagonals in the leftmost and rightmost squares that the loop passes through. For convenience, we call them \emph{start} and \emph{end squares.} This observation motivates the following structural parameter.

\begin{defn}[Structural parameter $\sigma_i$]
Given a $2$-by-$n$ triangular strip lattice, set $m=\lfloor n/2 \rfloor$. Let the type of diagonals that connect the top-left and bottom-right vertices of a square be denoted by ``$\backslash,$'' and the type of diagonals that connect the top-right and bottom-left vertices of the square be denoted by ``$/.$'' Define $\sigma_i$ as a parameter that encodes the symmetry of the diagonals in the start and end squares of the loop $\gamma_i$. If the $(m-i)$th square and the $(n-m+i)$th square have the same type of diagonals, i.e., $(\backslash,\backslash)$ or $(/,/)$, then $\sigma_i=1.$ Otherwise, if the squares have different types of diagonals, i.e., $(\backslash,/)$ or $(/,\backslash)$, then $\sigma_i = 0.$
\end{defn}

Using the indexing systems and the defined structural parameter, we obtain the following formulae.
\begin{theorem}\label{thm:Snforheterogeneous}
Let $S_{2n}$ denote the number of balanced spanning trees in $\Gamma_{2n}$. If $n$ is even, i.e., $n=2m$ for some $m$, then
\begin{align*}
S_{2n} = 2n-1+ 3T_{2m}^2 +\sum_{i=1}^{m-1}[\sigma_i((4i+1)T_{2(m-i)+1}^2+(4i+3)T_{2(m-i)}^2)\\
\quad +(1-\sigma_i)(8i+4)T_{2(m-i)+1}T_{2(m-i)}].
\end{align*}
If $n$ is odd, i.e., $n = 2m+1$ for some $m$, then
\begin{align*}
S_{2n} = 2n-1+ \sum_{i=0}^{m-1}[\sigma_i((4i+3)T_{2(m-i)+1}^2+(4i+5)T_{2(m-i)}^2)\\
\quad+(1-\sigma_i)(8i+8)T_{2(m-i)+1}T_{2(m-i)}].
\end{align*}
\end{theorem}
\par \medskip

\begin{example}\label{example:balancedspanningtrees} An example of how the formulae work in practice is as follows. Consider the graph $\Gamma_{20}$ in Figure \ref{fig:counting_balanced_trees}. Since $n=10$ is even, we apply the formula for $n$ even. First, we compute
\[2n-1 = 2(10)-1=19.\]
Next, observing that $m=\lfloor n/2 \rfloor = 5,$ and using the recurrence from Theorem \cref{thm:spanningtrees} or equation \ref{eqn:unifiedeqn} in the next section to compute $T_{2m}$, we obtain 
\[3T_{2m}^2=3T_{10}^2=3(2584)^2=20,031,168.\]
Finally, we calculate the value associated with each index $i$ in the summand. For $i=1,$ observe that the diagonals of the start square and end square are the same type, so $\sigma_1=1.$ Substituting $\sigma_i=1$ for index $i=1$ in the summand, we obtain 
\begin{align*}(4i+1)T_{2(m-i)+1}^2 + (4i+3)T_{2(m-i)}^2 &= 5T_{9}^2+7T_8^2
\\ &=5(987)^2+7(377)^2
\\ &=5,865,748.\end{align*}
For $i=2,$ notice that the diagonals in the start and end squares are different, so $\sigma_2=0.$ Substituting this into the $i=2$ term in the summand, we obtain
\begin{align*}
(8i+4)T_{2(m-i)+1}T_{2(m-i)}&=20T_7T_6\\
&=20(144)(55) \\
&=158,400.
\end{align*}
Repeat this process for the remaining indices: $i=3$ gives $4704,$ and $i=4$ gives $172.$ The total number of balanced spanning trees of $\Gamma_{20}$ equals the sum of all calculated results, which is $26,060,211.$

 
\end{example} 

\noindent \textbf{Proof of Theorem \ref{thm:Snforheterogeneous}.} We now prove Theorem \ref{thm:Snforheterogeneous}.

\begin{proof}[Even Case] 
    We first consider the even case where $n=2m$.\\
    \indent We begin with the symmetric partitions $i=0$ and $i=m$. For $i=0$, length$(\gamma_0)$ is always $3$, and $\gamma_0$ divides the graph into two end blocks, each consisting of $2m$ vertices. As illustrated in Figure $\ref{fig:counting_balanced_trees}$, $\gamma_0$ therefore contributes $3T_{2m}^2$ balanced spanning trees. For $i=m,$ the loop $\gamma_m$ has length $2n-1$. In addition, on each half of the $m$-th partition, there is only one possible spanning tree. Thus, the number of balanced spanning trees associated with $\gamma_m$ is $2n-1.$ See Figure \ref{fig:ex_n=10,i=0,i=m}.
    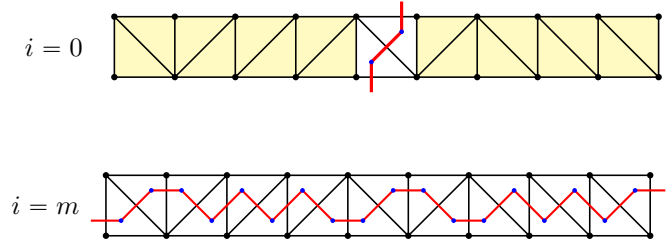
\begin{figure}[H]
    \centering
        \begin{tikzpicture}[semithick,xscale=0.8,yscale=0.8]
            \fill[yellow!30] (0,0) rectangle (4,1);
            \fill[yellow!30] (5,0) rectangle (9,1);
            \draw[] grid(9,1);
            \draw[] (0,1)--(1,0)--(2,1) (2,0)--(3,1) (3,0)--(4,1)--(5,0)--(6,1)--(7,0) (7,1)--(8,0) (8,1)--(9,0);
            \fill (0,0) circle (1.5pt);
            \fill (1,0) circle (1.5pt);
            \fill (2,0) circle (1.5pt);
            \fill (3,0) circle (1.5pt);
            \fill (4,0) circle (1.5pt);
            \fill (5,0) circle (1.5pt);
            \fill (6,0) circle (1.5pt);
            \fill (7,0) circle (1.5pt);
            \fill (8,0) circle (1.5pt);
            \fill (9,0) circle (1.5pt);
            \fill (0,1) circle (1.5pt);
            \fill (1,1) circle (1.5pt);
            \fill (2,1) circle (1.5pt);
            \fill (3,1) circle (1.5pt);
            \fill (4,1) circle (1.5pt);
            \fill (5,1) circle (1.5pt);
            \fill (6,1) circle (1.5pt);
            \fill (7,1) circle (1.5pt);
            \fill (8,1) circle (1.5pt);
            \fill (9,1) circle (1.5pt);
            \draw[red, very thick] (4.75,1.25)--(4.75,0.75)--(4.25,0.25)--(4.25,-0.25);
            \fill[blue] (4.25,0.25) circle (1.1pt);
            \fill[blue] (4.75,0.75) circle (1.1pt);
            \node at (-1, 0.5) {$i=0$};
        \end{tikzpicture}
    \\\vspace{1cm}
        \begin{tikzpicture}[semithick,xscale=0.8,yscale=0.8]
            \draw[] grid(9,1);
            \draw[] (0,1)--(1,0)--(2,1) (2,0)--(3,1) (3,0)--(4,1)--(5,0)--(6,1)--(7,0) (7,1)--(8,0) (8,1)--(9,0);
            \fill (0,0) circle (1.5pt);
            \fill (1,0) circle (1.5pt);
            \fill (2,0) circle (1.5pt);
            \fill (3,0) circle (1.5pt);
            \fill (4,0) circle (1.5pt);
            \fill (5,0) circle (1.5pt);
            \fill (6,0) circle (1.5pt);
            \fill (7,0) circle (1.5pt);
            \fill (8,0) circle (1.5pt);
            \fill (9,0) circle (1.5pt);
            \fill (0,1) circle (1.5pt);
            \fill (1,1) circle (1.5pt);
            \fill (2,1) circle (1.5pt);
            \fill (3,1) circle (1.5pt);
            \fill (4,1) circle (1.5pt);
            \fill (5,1) circle (1.5pt);
            \fill (6,1) circle (1.5pt);
            \fill (7,1) circle (1.5pt);
            \fill (8,1) circle (1.5pt);
            \fill (9,1) circle (1.5pt);
            \draw[red,thick] (-0.25,0.25)--(0.25,0.25)--(0.75,0.75)--(1.25,0.75)--(1.75,0.25)--(2.25,0.75)--(2.75,0.25)--(3.25,0.75)--(3.75,0.25)--(4.25,0.25)--(4.75,0.75)--(5.25,0.75)--(5.75,0.25)--(6.25,0.25)--(6.75,0.75)--(7.25,0.25)--(7.75,0.75)--(8.25,0.25)--(8.75,0.75)--(9.25,0.75);
            \fill[blue] (0.25,0.25) circle(1.1pt);
            \fill[blue] (0.75,0.75) circle(1.1pt);
            \fill[blue] (1.25,0.75) circle(1.1pt);
            \fill[blue] (1.75,0.25) circle(1.1pt);
            \fill[blue] (2.25,0.75) circle(1.1pt);
            \fill[blue] (2.75,0.25) circle(1.1pt);
            \fill[blue] (3.25,0.75) circle(1.1pt);
            \fill[blue] (3.75,0.25) circle(1.1pt);
            \fill[blue] (4.25,0.25) circle(1.1pt);
            \fill[blue] (4.75,0.75) circle(1.1pt);
            \fill[blue] (5.25,0.75) circle(1.1pt);
            \fill[blue] (5.75,0.25) circle(1.1pt);
            \fill[blue] (6.25,0.25) circle(1.1pt);
            \fill[blue] (6.75,0.75) circle(1.1pt);
            \fill[blue] (7.25,0.25) circle(1.1pt);
            \fill[blue] (7.75,0.75) circle(1.1pt);
            \fill[blue] (8.25,0.25) circle(1.1pt);
            \fill[blue] (8.75,0.75) circle(1.1pt);
            \node at (-1,0.5) {$i=m$};
        \end{tikzpicture}
        \caption{Example for $i=0$ and $i=m$ in a $2$-by-$10$ triangular strip lattice}
        \label{fig:ex_n=10,i=0,i=m}
    \end{figure}
    
    \indent For $i=1,\dots,m-1,$ the contribution depends on the diagonal configuration. There are two cases: $\sigma_i = 1$ and $\sigma_i = 0.$ 

    Recall that if $\sigma_i=1$, the start and end squares of a valid loop contain the same type of diagonal. There are two possible configurations, namely $(\backslash,\backslash)$ and $(/,/).$ We consider the first configuration and determine the lengths of $\gamma_i$ and $\mu_i$ together with the sizes of their end blocks.

    Observe that a $2$-by-$n$ triangular strip lattice contains $n-1$ squares. For $i\ne m,$ both $\gamma_i$ and $\mu_i$ pass through the start square $m-i,$ the end square $n-m+i,$ and every square between them. We refer to the squares strictly between the start and end squares as the \emph{middle squares} (highlighted in pink in Figure \ref{fig:ex_n=10,i=2}). The number of the middle squares is $(n-m+i)-(m-i)-1=2i-1.$ 
    Every middle square contributes two edges to each valid loop, except for one of them, which contributes three edges. Hence the middle squares collectively contribute $2(2i-1)+1=4i-1$ edges to each valid loop. Since $\gamma_i$ does not intersect the diagonals of the start and end squares in this configuration, the start and end squares each contribute one additional edge. Meanwhile, since $\mu_i$ intersects the diagonals of the start and end squares, the start and end squares each contribute two additional edges. From this paragraph, we obtain
    \[\text{length}(\gamma_i)=4i+1,\]
    \[\text{length}(\mu_i)=4i+3.\]

    Next, observe that each end block (highlighted in yellow in Figure \ref{fig:ex_n=10,i=2}) contains vertices belonging to the first $m-i-1$ complete squares. Since $\gamma_i$ does not intersect the diagonals of the start and end squares, each of its end blocks contains an extra vertex. Therefore, the end blocks of $\gamma_i$ contain $2(m-i)+1$ vertices each, and the end blocks of $\mu_i$ contains $2(m-i)$ vertices each. Together, $\gamma_i$ and $\mu_i$ contribute
    \[(4i+1)T_{2(m-i)+1}^2+(4i+3)T_{2(m-i)}^2\]
    balanced spanning trees.
    
    \begin{figure}[H]
        \centering
        \begin{tikzpicture}[semithick,xscale=0.8,yscale=0.8]
            \fill[pink!60] (3,0) rectangle (6,1);
            \fill[yellow!30] (0,0)--(3,0)--(2,1)--(0,1)--cycle;
            \fill[yellow!30] (6,1)--(7,0)--(9,0)--(9,1)--cycle;
            \draw[] grid(9,1);
            \draw (2,1)--(3,0) (6,1)--(7,0);
            \draw (3,0)--(4,1)--(5,0)--(6,1);
            \draw[red, very thick] (2.75,1.3)--(2.75,0.75)--(3.25,0.75)--(3.75,0.25)--(4.25,0.25)--(4.75,0.75)--(5.25,0.75)--(5.75,0.25)--(6.25,0.25)--(6.25,-0.3);
            \fill[blue] (2.75,0.75) circle(1.1pt);
            \fill[blue] (3.25,0.75) circle(1.1pt);
            \fill[blue] (3.75,0.25) circle(1.1pt);
            \fill[blue] (4.25,0.25) circle(1.1pt);
            \fill[blue] (4.75,0.75) circle(1.1pt);
            \fill[blue] (5.25,0.75) circle(1.1pt);
            \fill[blue] (5.75,0.25) circle(1.1pt);
            \fill[blue] (6.25,0.25) circle(1.1pt);
            \node[red] at (2.75,1.6) {$\gamma_i$};

            nodes
            \fill (0,0) circle (1.5pt);
            \fill (1,0) circle (1.5pt);
            \fill (2,0) circle (1.5pt);
            \fill (3,0) circle (1.5pt);
            \fill (4,0) circle (1.5pt);
            \fill (5,0) circle (1.5pt);
            \fill (6,0) circle (1.5pt);
            \fill (7,0) circle (1.5pt);
            \fill (8,0) circle (1.5pt);
            \fill (9,0) circle (1.5pt);
            \fill (0,1) circle (1.5pt);
            \fill (1,1) circle (1.5pt);
            \fill (2,1) circle (1.5pt);
            \fill (3,1) circle (1.5pt);
            \fill (4,1) circle (1.5pt);
            \fill (5,1) circle (1.5pt);
            \fill (6,1) circle (1.5pt);
            \fill (7,1) circle (1.5pt);
            \fill (8,1) circle (1.5pt);
            \fill (9,1) circle (1.5pt);
        \end{tikzpicture}
        \\\vspace{0cm}
        \begin{tikzpicture}[semithick,xscale=0.8,yscale=0.8]
            \fill[pink!60] (3,0) rectangle (6,1);
            \fill[yellow!30] (0,0)--(2,0)--(2,1)--(0,1)--cycle;
            \fill[yellow!30] (7,1)--(7,0)--(9,0)--(9,1)--cycle;
            \draw[] grid(9,1);
            \draw (2,1)--(3,0) (6,1)--(7,0);
            \draw (3,0)--(4,1)--(5,0)--(6,1);
            \draw[red, very thick] (2.25,-0.3)--(2.25,0.25)--(2.75,0.75)--(3.25,0.75)--(3.75,0.25)--(4.25,0.25)--(4.75,0.75)--(5.25,0.75)--(5.75,0.25)--(6.25,0.25)--(6.75,0.75)--(6.75,1.3);
            \fill[blue] (2.25,0.25) circle(1.1pt);
            \fill[blue] (2.75,0.75) circle(1.1pt);
            \fill[blue] (3.25,0.75) circle(1.1pt);
            \fill[blue] (3.75,0.25) circle(1.1pt);
            \fill[blue] (4.25,0.25) circle(1.1pt);
            \fill[blue] (4.75,0.75) circle(1.1pt);
            \fill[blue] (5.25,0.75) circle(1.1pt);
            \fill[blue] (5.75,0.25) circle(1.1pt);
            \fill[blue] (6.25,0.25) circle(1.1pt);
            \fill[blue] (6.75,0.75) circle(1.1pt);
            \node[red] at (6.75,1.6) {$\mu_i$};

            nodes
            \fill (0,0) circle (1.5pt);
            \fill (1,0) circle (1.5pt);
            \fill (2,0) circle (1.5pt);
            \fill (3,0) circle (1.5pt);
            \fill (4,0) circle (1.5pt);
            \fill (5,0) circle (1.5pt);
            \fill (6,0) circle (1.5pt);
            \fill (7,0) circle (1.5pt);
            \fill (8,0) circle (1.5pt);
            \fill (9,0) circle (1.5pt);
            \fill (0,1) circle (1.5pt);
            \fill (1,1) circle (1.5pt);
            \fill (2,1) circle (1.5pt);
            \fill (3,1) circle (1.5pt);
            \fill (4,1) circle (1.5pt);
            \fill (5,1) circle (1.5pt);
            \fill (6,1) circle (1.5pt);
            \fill (7,1) circle (1.5pt);
            \fill (8,1) circle (1.5pt);
            \fill (9,1) circle (1.5pt);
        \end{tikzpicture}
        \caption{Example for $i=2$ in a $2$-by-$10$ triangular strip lattice ($\sigma_i=1$)}
        \label{fig:ex_n=10,i=2}
    \end{figure}
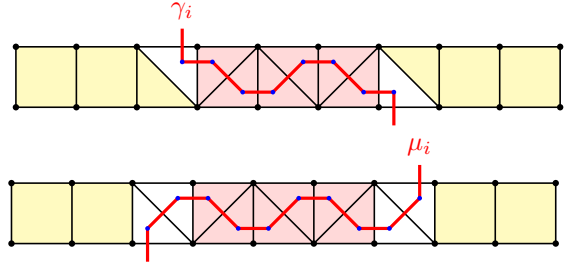
    
    In the other configuration $(/,/),$ the roles of $\gamma_i$ and $\mu_i$ are simply reversed: $\gamma_i$ has length $4i+3$ and end blocks of size $2(m-i),$ while $\mu_i$ has length $4i+3$ and end blocks of size $2(m-i).$ Hence, both configurations contribute the same number of balanced spanning trees. 
    
    For $\sigma_i=0,$ the possible configurations are $(/,\backslash)$ and $(\backslash,/).$ In either configuration, $\gamma_i$ and $\mu_i$ each intersects only one of the diagonals of the start and end cells. (See Figure \ref{fig:ex_n=10,i=2,sigma=0}.) Therefore,

    \[\text{length}(\gamma_i)=\text{length}(\mu_i)=4i+2.\]
    Furthermore, both $\gamma_i$ and $\mu_i$ have end blocks containing $2(m-i)+1$ and $2(m-i)$ vertices. Therefore, $\sigma_i=0$ contributes 
    \[(8i+4)T_{2(m-i)+1}T_{2(m-i)}\]
    balanced spanning trees.
        \begin{figure}[H]
        \centering
        \begin{tikzpicture}[semithick,xscale=0.8,yscale=0.8]
            \fill[yellow!30] (0,0)--(3,0)--(2,1)--(0,1)--cycle;
            \fill[yellow!30] (7,1)--(7,0)--(9,0)--(9,1)--cycle;
            \draw[] grid(9,1);
            \draw (2,1)--(3,0) (6,0)--(7,1);
            \draw (3,0)--(4,1)--(5,0)--(6,1);
            \draw[red, very thick] (2.75,1.3)--(2.75,0.75)--(3.25,0.75)--(3.75,0.25)--(4.25,0.25)--(4.75,0.75)--(5.25,0.75)--(5.75,0.25)--(6.25,0.75)--(6.75,0.25)--(6.75,-0.3);
            \fill[blue] (2.75,0.75) circle(1.1pt);
            \fill[blue] (3.25,0.75) circle(1.1pt);
            \fill[blue] (3.75,0.25) circle(1.1pt);
            \fill[blue] (4.25,0.25) circle(1.1pt);
            \fill[blue] (4.75,0.75) circle(1.1pt);
            \fill[blue] (5.25,0.75) circle(1.1pt);
            \fill[blue] (5.75,0.25) circle(1.1pt);
            \fill[blue] (6.25,0.75) circle(1.1pt);
            \fill[blue] (6.75,0.25) circle(1.1pt);
            \node[red] at (2.75,1.6) {$\gamma_i$};

            nodes
            \fill (0,0) circle (1.5pt);
            \fill (1,0) circle (1.5pt);
            \fill (2,0) circle (1.5pt);
            \fill (3,0) circle (1.5pt);
            \fill (4,0) circle (1.5pt);
            \fill (5,0) circle (1.5pt);
            \fill (6,0) circle (1.5pt);
            \fill (7,0) circle (1.5pt);
            \fill (8,0) circle (1.5pt);
            \fill (9,0) circle (1.5pt);
            \fill (0,1) circle (1.5pt);
            \fill (1,1) circle (1.5pt);
            \fill (2,1) circle (1.5pt);
            \fill (3,1) circle (1.5pt);
            \fill (4,1) circle (1.5pt);
            \fill (5,1) circle (1.5pt);
            \fill (6,1) circle (1.5pt);
            \fill (7,1) circle (1.5pt);
            \fill (8,1) circle (1.5pt);
            \fill (9,1) circle (1.5pt);
        \end{tikzpicture}
        \\\vspace{0cm}
        \begin{tikzpicture}[semithick,xscale=0.8,yscale=0.8]
            \fill[yellow!30] (0,0)--(2,0)--(2,1)--(0,1)--cycle;
            \fill[yellow!30] (7,1)--(6,0)--(9,0)--(9,1)--cycle;
            \draw[] grid(9,1);
            \draw (2,1)--(3,0) (6,0)--(7,1);
            \draw (3,0)--(4,1)--(5,0)--(6,1);
            \draw[red, very thick] (2.25,-0.3)--(2.25,0.25)--(2.75,0.75)--(3.25,0.75)--(3.75,0.25)--(4.25,0.25)--(4.75,0.75)--(5.25,0.75)--(5.75,0.25)--(6.25,0.75)--(6.25,1.3);
            \fill[blue] (2.75,0.75) circle(1.1pt);
            \fill[blue] (3.25,0.75) circle(1.1pt);
            \fill[blue] (3.75,0.25) circle(1.1pt);
            \fill[blue] (4.25,0.25) circle(1.1pt);
            \fill[blue] (4.75,0.75) circle(1.1pt);
            \fill[blue] (5.25,0.75) circle(1.1pt);
            \fill[blue] (5.75,0.25) circle(1.1pt);
            \fill[blue] (6.25,0.75) circle(1.1pt);
            \node[red] at (6.25,1.6) {$\mu_i$};

            nodes
            \fill (0,0) circle (1.5pt);
            \fill (1,0) circle (1.5pt);
            \fill (2,0) circle (1.5pt);
            \fill (3,0) circle (1.5pt);
            \fill (4,0) circle (1.5pt);
            \fill (5,0) circle (1.5pt);
            \fill (6,0) circle (1.5pt);
            \fill (7,0) circle (1.5pt);
            \fill (8,0) circle (1.5pt);
            \fill (9,0) circle (1.5pt);
            \fill (0,1) circle (1.5pt);
            \fill (1,1) circle (1.5pt);
            \fill (2,1) circle (1.5pt);
            \fill (3,1) circle (1.5pt);
            \fill (4,1) circle (1.5pt);
            \fill (5,1) circle (1.5pt);
            \fill (6,1) circle (1.5pt);
            \fill (7,1) circle (1.5pt);
            \fill (8,1) circle (1.5pt);
            \fill (9,1) circle (1.5pt);
        \end{tikzpicture}
        \caption{Example for $i=2$ in a $2$-by-$10$ triangular strip lattice ($\sigma_i=0$)}
        \label{fig:ex_n=10,i=2,sigma=0}
    \end{figure}
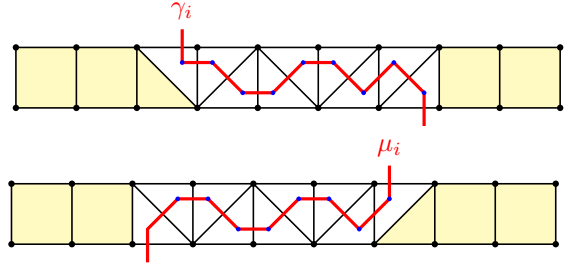
    Summing over all indices completes the proof for the even case.
\end{proof}

\begin{proof}[Odd Case]
    The proof for the odd case is analogous to the even case. We first consider the symmetric partition $i=m.$ Similar to the even case, length$(\gamma_m)=2n-1,$ and on either side of $\gamma_m,$ there is only one possible spanning tree. (See Figure \ref{fig:ex_n=11,i=m}.) Therefore, $\gamma_m$ contributes $2n-1$ balanced spanning trees.
    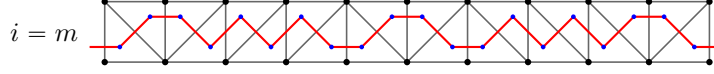
\begin{figure}[H]
        \begin{tikzpicture}[semithick,xscale=0.8,yscale=0.8]
            \draw[.!60] grid(10,1);
            \draw[.!60] (0,1)--(1,0)--(2,1) (2,0)--(3,1) (3,0)--(4,1)--(5,0)--(6,1)--(7,0) (7,1)--(8,0) (8,1)--(9,0)--(10,1);
            \fill (0,0) circle (1.5pt);
            \fill (1,0) circle (1.5pt);
            \fill (2,0) circle (1.5pt);
            \fill (3,0) circle (1.5pt);
            \fill (4,0) circle (1.5pt);
            \fill (5,0) circle (1.5pt);
            \fill (6,0) circle (1.5pt);
            \fill (7,0) circle (1.5pt);
            \fill (8,0) circle (1.5pt);
            \fill (9,0) circle (1.5pt);
            \fill (10,0) circle (1.5pt);
            \fill (0,1) circle (1.5pt);
            \fill (1,1) circle (1.5pt);
            \fill (2,1) circle (1.5pt);
            \fill (3,1) circle (1.5pt);
            \fill (4,1) circle (1.5pt);
            \fill (5,1) circle (1.5pt);
            \fill (6,1) circle (1.5pt);
            \fill (7,1) circle (1.5pt);
            \fill (8,1) circle (1.5pt);
            \fill (9,1) circle (1.5pt);
            \fill (10,1) circle (1.5pt);
            
            \draw[red,thick] (-0.25,0.25)--(0.25,0.25)--(0.75,0.75)--(1.25,0.75)--(1.75,0.25)--(2.25,0.75)--(2.75,0.25)--(3.25,0.75)--(3.75,0.25)--(4.25,0.25)--(4.75,0.75)--(5.25,0.75)--(5.75,0.25)--(6.25,0.25)--(6.75,0.75)--(7.25,0.25)--(7.75,0.75)--(8.25,0.25)--(8.75,0.75)--(9.25,0.75)--(9.75,0.25)--(10.25,0.25);
            \fill[blue] (0.25,0.25) circle(1.1pt);
            \fill[blue] (0.75,0.75) circle(1.1pt);
            \fill[blue] (1.25,0.75) circle(1.1pt);
            \fill[blue] (1.75,0.25) circle(1.1pt);
            \fill[blue] (2.25,0.75) circle(1.1pt);
            \fill[blue] (2.75,0.25) circle(1.1pt);
            \fill[blue] (3.25,0.75) circle(1.1pt);
            \fill[blue] (3.75,0.25) circle(1.1pt);
            \fill[blue] (4.25,0.25) circle(1.1pt);
            \fill[blue] (4.75,0.75) circle(1.1pt);
            \fill[blue] (5.25,0.75) circle(1.1pt);
            \fill[blue] (5.75,0.25) circle(1.1pt);
            \fill[blue] (6.25,0.25) circle(1.1pt);
            \fill[blue] (6.75,0.75) circle(1.1pt);
            \fill[blue] (7.25,0.25) circle(1.1pt);
            \fill[blue] (7.75,0.75) circle(1.1pt);
            \fill[blue] (8.25,0.25) circle(1.1pt);
            \fill[blue] (8.75,0.75) circle(1.1pt);
            \fill[blue] (9.25,0.75) circle(1.1pt);
            \fill[blue] (9.75,0.25) circle(1.1pt);
            \node at (-1,0.5) {$i=m$};
        \end{tikzpicture}
        \caption{Example for $i=m$ in a $2$-by-$11$ triangular strip lattice}
        \label{fig:ex_n=11,i=m}
    \end{figure}
    Next, we consider the cases when $\sigma_i=1$ and $\sigma_i=0$ for non-symmetric indices $0,1,\dots,m-1.$ Here, we apply the exact same strategy used in the even case. The only difference is that the middle squares here contribute $4i+1$ edges to each loop. Adjusting for this discrepancy and treating $i=0$ as a non-symmetric partition, we achieve the formula for the odd case.

\end{proof}


\subsection{Maximizing and minimizing $S_{2n}$} We are also interested in figuring out the configurations for $\Gamma_{2n}$ so that $S_{2n}$ is maximized and minimized. Our strategy is to find a relationship between the parameter $\sigma_i$ and the number of spanning trees associated with $i$. In doing so, we make the following observation.
\begin{lemma}\label{lemma:regarding_sigmas}
    For every non-symmetric partition pair $i$ of $\Gamma_{2n},$ the number of balanced spanning trees contributed by $i$ is maximized when $\sigma_i = 1$ and minimized when $\sigma_i=0$. 
\end{lemma}

\begin{proof}
    For even $n$ and $1 \leq i \leq m-1,$ the contributions corresponding to $\sigma_i=1$ and $\sigma_i=0,$ are $(4i+1)T_{2(m-i)+1}^2+(4i+3)T_{2(m-i)}^2$ and $(8i+4)T_{2(m-i)+1}T_{2(m-i)},$ respectively. Subtracting the latter from the former, we obtain
    \begin{align*}&(4i+1)T_{2(m-i)+1}^2+(4i+3)T_{2(m-i)}^2-(8i+4)T_{2(m-i)+1}T_{2(m-i)}\\
    &=(4i+1)(T_{2(m-i)+1}-T_{2(m-i)})^2-2T_{2(m-i)}(T_{2(m-i)+1}-T_{2(m-i)}).
    \end{align*}
    Substituting the identity $T_{2(m-i)+1}=2T_{2(m-i)}+F_{2(m-i)}$ yields
    \begin{align*}&(4i+1)(T_{2(m-i)}+F_{2(m-i)})^2-2T_{2(m-i)}(T_{2(m-i)}+F_{2(m-i)})\\
    &=4(i-1)(T_{2(m-i)}+F_{2(m-i)})^2+3T_{2(m-i)}^2+8T_{2(m-i)}F_{2(m-i)}+5F_{2(m-i)}^2.
    \end{align*}
    Since $i \geq 1$ and $T_{2(m-i)},F_{2(m-i)} > 0,$ the expression above is strictly positive. Hence, 
    \[(4i+1)T_{2(m-i)+1}^2+(4i+3)T_{2(m-i)}^2 > (8i+4)T_{2(m-i)+1}T_{2(m-i)}.\] This completes the proof for the even case.

    The proof for odd $n$ is analogous. For $0 \leq i \leq m-1,$ subtracting the contribution corresponding to $\sigma_i=0$ from the contribution corresponding to $\sigma_i=1,$ we obtain
    \[(4i+3)T_{2(m-i)+1}^2+(4i+5)T_{2(m-i)}^2-(8i+8)T_{2(m-i)+1}T_{2(m-i)}.\]
    Substituting $T_{2(m-i)+1}=2T_{2(m-i)}+F_{2(m-i)}$ and following the same computation as in the even case yields
    \[4i(T_{2(m-i)}+F_{2(m-i)})^2 + T_{2(m-i)}^2 + 4T_{2(m-i)}F_{2(m-i)} + 3F_{2(m-i)}^2.\]
    This expression is also strictly positive. Hence, the contribution corresponding to $\sigma_i=1$ is greater than that corresponding to $\sigma_i=0.$
    \end{proof}


The corollary below immediately follows Lemma $\ref{lemma:regarding_sigmas}.$
\begin{corollary}\label{corollary:maximizeminimize}
    $S_{2n}$ is maximized when $\sigma_i=1$ for all $i$ and minimized when $\sigma_i=0$ for all $i.$ 
\end{corollary}

\section{Calculating the proportion of balanced spanning trees} \label{sec:proportion}

To calculate the proportion of balanced spanning trees, we first figure out a closed-form formula for $T_{k}$, the number of spanning trees in a graph with $k$ vertices. We have 2 cases depending on the parity of $k$. 
\par 
If $k=2n$, using (\ref{eqn:recurrencerelationspanningtree})
and (\ref{eqn:initialvalues1}), we get the following: 
\begin{equation}\label{eqn:evencaseformula}
    T_{2n} = \beta r^n - \theta r^{-n}  \quad \text{where} \quad 
    \beta = \frac{3\sqrt{5} - 5}{10}, \quad \theta = \frac{3\sqrt{5} + 5}{10} \quad \text{and} \quad r = \frac{7 + 3\sqrt{5}}{2}.
\end{equation}

Similarly, $k=2n+1$, using (\ref{eqn:recurrencerelationspanningtreeodd})
and (\ref{eqn:initialvalues2}), we get the following:  

\begin{equation}\label{eqn:oddcaseformula}
    T_{2n + 1} = \alpha (r^n - r^{-n} ) \quad \text{where} \quad 
    \alpha = \frac{1}{\sqrt{5}} \quad \text{and} \quad r = \frac{7 + 3\sqrt{5}}{2}.
\end{equation}
We can unify both equations and obtain the following: 

\begin{equation}\label{eqn:unifiedeqn}
    T_k = \alpha (\phi^{2k - 2} - \phi^{-(2k - 2)}) \quad \text{where} \quad \alpha = \frac{1}{\sqrt{5}}, \quad \phi = \frac{\sqrt{5} + 1}{2} .
\end{equation}

\subsection{Finding the upper bound}

The main goal of this section is to prove the following: \begin{proposition}\label{prop:upperbound}
For any $n \geq 1$, if $\sigma_i = 1$ for all $i \in \{1, 2, \dots, m-1\}$, the following holds:

\[\lim_{n\to \infty}\frac{S_{2n}}{T_{2n}} = 0.668328.\]
\end{proposition}.

We deal with the case where $\sigma_i$ = $1$ for all $i$ $\in$ $\{1, 2, 3, \dots, m - 1\}$. Recall by Corollary~\ref{corollary:maximizeminimize} that this gives the maximum $S_{2n}$. By \Cref{thm:Snforheterogeneous}, in our case, if $n = 2m$, then, \[S_{2n} = (2n-1) + 3T_{2m}^2 +\sum_{i=1}^{m-1}((4i+1)T_{2(m-i)+1}^2+(4i+3)T_{2(m-i)}^2).    
\]
Observe that, 

\begin{align*} \sum_{i=1}^{m-1}((4i+1)T_{2(m-i)+1}^2+(4i+3)T_{2(m-i)}^2) &= \sum_{i = 1}^{m - 1}(4i+1)T_{2(m-i)+1}^2 + \sum_{i = 1}^{m - 1}(4i+3)T_{2(m-i)}^2\\ 
&= \sum_{j = 2; j \text{ is even}}^{2(m - 1)}(2j+1)T_{n+1-j}^2 + \sum_{j = 3; j \text{ is odd}}^{2m - 1}(2j+1)T_{n+1-j}^2 \\
&= \sum_{i = 2}^{n - 1}(2i + 1)T^2_{n + 1 - i}.
\end{align*}

And if $n = 2m + 1$ is odd, then, 

\[S_{2n}= (2n-1)+ \sum_{i=0}^{m-1}((4i+3)T_{2(m-i)+1}^2+(4i+5)T_{2(m-i)}^2).\]

Observe that, 

\begin{align*} \sum_{i=0}^{m-1}((4i+3)T_{2(m-i)+1}^2+(4i+5)T_{2(m-i)}^2) &= \sum_{i=0}^{m-1}(4i+3)T_{2(m-i)+1}^2 + \sum_{i=0}^{m-1}(4i+5)T_{2(m-i)}^2\\ 
&= \sum_{j = 2; j \text{ is even}}^{2m}(2j+1)T_{n+1-j}^2 + \sum_{j = 1; j \text{ is odd}}^{2m - 1}(2j+1)T_{n+1-j}^2 \\
&= \sum_{i = 1}^{n - 1}(2i + 1)T^2_{n + 1 - i}.
\end{align*}



Now, with some substitutions, we can unify both formulae to obtain the following: 

\[S_{2n} = (2n - 1) + \sum_{i = 1}^{n - 1} (2i + 1)T_{n + 1 - i}^2.\]

This equation applies to both even and odd $n$.

\noindent \textit{Proof of Proposition \ref{prop:upperbound}.}
    Observe that, for $i \in \{1,2,\dots,n\} $, 
    \begin{equation} \label{eqn:simplifyfrac}
    \begin{aligned}
    \frac{T_{n + 1 - i}^2}{T_{2n}} 
    & \stackrel{(\ref{eqn:unifiedeqn})}{=} \frac{\alpha^2(\phi^{4n - 4i} + \phi^{-(4n - 4i)} - 2)}{\alpha(\phi^{4n - 2} - \phi^{-(4n - 2)})} \\
    & = \alpha \frac{\phi^{-4i} + \phi^{-8n + 4i} - 2\phi^{-4n}}{\phi^{-2} - \phi^{-8n + 2}}.
    \end{aligned}
    \end{equation}
Also, note the following inequality: 

\begin{equation}\label{eqn:inequalityfirst}
    T_{a + b} > 3T_a \cdot T_b.
\end{equation}
This can be shown algebraically by using (\ref{eqn:unifiedeqn}) or by noting that, if we consider two spanning trees $\mathcal{T}_a$ and $\mathcal{T}_b$ and we align them to form a forest,  there are 3 ways to add an edge to get a spanning tree $\mathcal{T}_{a + b}$. Furthermore, there are more ways of getting a spanning tree $\mathcal{T}_{a + b}$ without joining two trees. This works whenever one of $a$ or $b$ is even when we place the spanning tree with even vertices on the left. See \Cref{fig:ineqexample1} for an example.


 \begin{figure}[H]
        \centering
        \begin{tikzpicture}[scale=1, line join=round, line cap=round]
    
    \definecolor{gridgray}{RGB}{235, 235, 235}
    \definecolor{primalred}{RGB}{234, 67, 72}
    \definecolor{edgeblue}{RGB}{66, 133, 244}

    \begin{scope}[yshift=1.8cm]
        
        \foreach \y in {0,1} {
            \draw[gridgray, line width=1.2pt] (0,\y) -- (5,\y);
        }
        \foreach \x in {0,...,5} {
            \draw[gridgray, line width=1.2pt] (\x,0) -- (\x,1);
        }
        
        \draw[gridgray, line width=1.2pt] (0,0) -- (1,1);
        \draw[gridgray, line width=1.2pt] (1,0) -- (2,1);
        \draw[gridgray, line width=1.2pt] (2,1) -- (3,0);
        \draw[gridgray, line width=1.2pt] (3,0) -- (4,1);
        \draw[gridgray, line width=1.2pt] (4,1) -- (5,0);

        \draw[black, line width=1.5pt] (0,1) -- (1,1);
        \draw[black, line width=1.5pt] (0,0) -- (1,0);
        \draw[black, line width=1.5pt] (1,0) -- (1,1);
        \draw[black, line width=1.5pt] (1,0) -- (2,1);
        \draw[black, line width=1.5pt] (2,0) -- (2,1);
        
        \draw[black, line width=1.5pt] (3,1) -- (4,1);
        \draw[black, line width=1.5pt] (3,0) -- (4,0);
        \draw[black, line width=1.5pt] (3,0) -- (4,1);
        
        \draw[black, line width=1.5pt] (4,1) -- (5,0);
        \draw[black, line width=1.5pt] (5,0) -- (5,1);

        \draw[edgeblue, line width=1.5pt] (2,1) -- (3,1);
        \draw[edgeblue, line width=1.5pt] (2,1) -- (3,0);
        \draw[edgeblue, line width=1.5pt] (2,0) -- (3,0);

        \foreach \x in {0,...,5} {
            \foreach \y in {0,1} {
                \fill[primalred] (\x, \y) circle (2.5pt);
            }
        }
    \end{scope}

    \begin{scope}[yshift=0cm]
        
        \foreach \y in {0,1} {
            \draw[gridgray, line width=1.2pt] (0,\y) -- (5,\y);
        }
        \foreach \x in {0,...,5} {
            \draw[gridgray, line width=1.2pt] (\x,0) -- (\x,1);
        }
        
        \draw[gridgray, line width=1.2pt] (0,0) -- (1,1);
        \draw[gridgray, line width=1.2pt] (1,0) -- (2,1);
        \draw[gridgray, line width=1.2pt] (2,1) -- (3,0);
        \draw[gridgray, line width=1.2pt] (3,0) -- (4,1);
        \draw[gridgray, line width=1.2pt] (4,1) -- (5,0);

        \draw[black, line width=1.5pt] (0,1) -- (1,1);
        \draw[black, line width=1.5pt] (0,0) -- (1,0);
        \draw[black, line width=1.5pt] (1,0) -- (1,1);
        \draw[black, line width=1.5pt] (1,0) -- (2,1);
        
        \draw[black, line width=1.5pt] (2,1) -- (3,1);
        \draw[black, line width=1.5pt] (2,0) -- (3,0);
        
        \draw[green, line width=1.5pt] (3,1) -- (4,1);
        \draw[black, line width=1.5pt] (3,0) -- (4,0);
        \draw[black, line width=1.5pt] (3,0) -- (4,1);
        
        \draw[black, line width=1.5pt] (4,1) -- (5,0);
        \draw[black, line width=1.5pt] (5,0) -- (5,1);

        \foreach \x in {0,...,5} {
            \foreach \y in {0,1} {
                \fill[primalred] (\x, \y) circle (2.5pt);
            }
        }
    \end{scope}

\end{tikzpicture}
        \caption{Top: Two spanning trees in $\Gamma_6$ (with black edges) that can be connected in 3 ways (the blue edges) to form a spanning tree in $\Gamma_{12}$. Bottom: A way to get a spanning tree in $\Gamma_{12}$ without joining two trees (the edges in the left 6 vertices don't form a spanning tree for $\Gamma_6$).}
        \label{fig:ineqexample1}
    \end{figure}
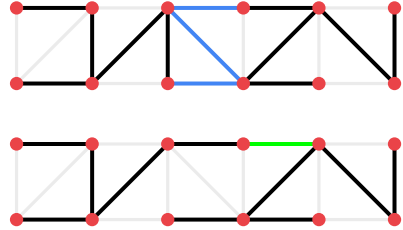
When both $a$ and $b$ are odd, we can align them to create a graph in $\Gamma_{a + b}$. We still get three ways to create a spanning tree in $\Gamma_{a + b}$. Furthermore, there are more ways of getting a spanning tree $\mathcal{T}_{a + b}$ without joining two trees in $\mathcal{T}_a$ and $\mathcal{T}_b$. See \Cref{fig:ineqexample2} for an example.

 \begin{figure}[H]
        \centering
        \begin{tikzpicture}[scale=1, line join=round, line cap=round]
    
    \definecolor{gridgray}{RGB}{235, 235, 235}
    \definecolor{primalred}{RGB}{234, 67, 72}
    \definecolor{edgeblue}{RGB}{66, 133, 244}

    \begin{scope}[yshift=1.8cm]
        
        \foreach \y in {0,1} {
            \draw[gridgray, line width=1.2pt] (0,\y) -- (5,\y);
        }
        \foreach \x in {0,...,5} {
            \draw[gridgray, line width=1.2pt] (\x,0) -- (\x,1);
        }
        
        \draw[gridgray, line width=1.2pt] (0,0) -- (1,1);
        \draw[gridgray, line width=1.2pt] (2,0) -- (3,1);

        \draw[black, line width=1.5pt] (0,1) -- (1,1);
        \draw[black, line width=1.5pt] (0,0) -- (1,0);
        \draw[black, line width=1.5pt] (1,0) -- (1,1);
        
        \draw[black, line width=1.5pt] (1,0) -- (2,1);
        \draw[black, line width=1.5pt] (2,0) -- (2,1);
        
        \draw[black, line width=1.5pt] (2,1) -- (3,1);
        
        \draw[black, line width=1.5pt] (3,0) -- (4,1);
        \draw[black, line width=1.5pt] (3,0) -- (4,0);
        
        \draw[black, line width=1.5pt] (4,1) -- (5,0);
        \draw[black, line width=1.5pt] (5,0) -- (5,1);

        \draw[edgeblue, line width=1.5pt] (2,0) -- (3,0);
        \draw[edgeblue, line width=1.5pt] (3,0) -- (3,1);
        \draw[edgeblue, line width=1.5pt] (3,1) -- (4,1);

        \foreach \x in {0,...,5} {
            \foreach \y in {0,1} {
                \fill[primalred] (\x, \y) circle (2.5pt);
            }
        }
    \end{scope}

    \begin{scope}[yshift=0cm]
        
        \foreach \y in {0,1} {
            \draw[gridgray, line width=1.2pt] (0,\y) -- (5,\y);
        }
        \foreach \x in {0,...,5} {
            \draw[gridgray, line width=1.2pt] (\x,0) -- (\x,1);
        }
        
        \foreach \x in {0,...,4} {
            \draw[gridgray, line width=1.2pt] (\x,0) -- (\x+1,1);
        }

        \draw[black, line width=1.5pt] (0,1) -- (5,1);
        \draw[black, line width=1.5pt] (0,0) -- (5,0);
        \draw[black, line width=1.5pt] (1,0) -- (1,1);

        \foreach \x in {0,...,5} {
            \foreach \y in {0,1} {
                \fill[primalred] (\x, \y) circle (2.5pt);
            }
        }
    \end{scope}

\end{tikzpicture}
        \caption{Top: Two spanning trees in $\Gamma_7$ on the left and $\Gamma_5$ on the right (with black edges) that can be connected in 3 ways (the blue edges) to form a spanning tree in $\Gamma_{12}$. Bottom: A way to get a spanning tree in $\Gamma_{12}$ without joining two trees in $\Gamma_7$ and $\Gamma_5$ (the edges in the right 5 vertices don't form a spanning tree for $\Gamma_5$).}
        \label{fig:ineqexample2}
    \end{figure}
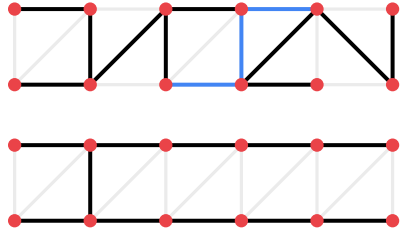
So in any case, this inequality holds. We can use this inequality to get an inequality that will be more useful to us: 

\begin{equation}\label{eqn:inequality}
    T_{a + b + c} > 9T_a \cdot T_b \cdot T_c
\end{equation}

Using this inequality, we get

\begin{equation}\label{eqn:inequality1}
   T_{2n} = T_{(n + 1 - i) + (n + 1 - i) + (2i - 2)} > 9T^2_{n + 1 - i} T_{2i - 2} .
\end{equation}
Now, for $i \geq 4$, we get the following,

\begin{align*}\frac{(2i + 1)T_{n + 1 - i}^2}{T_{2n}} &\stackrel{(\ref{eqn:inequality1})}{<} \frac{2i + 1}{9T_{2i - 2}}\\
&\stackrel{(\ref{eqn:unifiedeqn})}{=} \frac{2i + 1}{9\alpha (\phi^{4i - 6} - \phi^{-(4i - 6)})}\\
&< \frac{\sqrt{5}}{9}\frac{2i + 1}{(\frac{3}{2})^{4i - 6}} = \frac{\sqrt{5}}{9}(2i + 1)\bigg(\frac{2}{3}\bigg)^{4i - 6}.
\end{align*}
With this fact, we can use the dominated convergence theorem to establish the following:

\begin{equation}\label{eqn:dominatedconv}
    \lim_{n \to \infty} \sum_{i = 1}^{n-1} \frac{(2i + 1)T_{n+1-i}^2}{T_{2n}} =  \sum_{i = 1}^{\infty} \lim_{n \to \infty} \frac{(2i + 1)T_{n+1-i}^2}{T_{2n}} = \sum_{i = 1}^{\infty} (2i + 1) \lim_{n \to \infty} \frac{T_{n+1-i}^2}{T_{2n}}.
\end{equation}

\noindent where $T_k$ = 0 if $k < 0$. Now, observe that, 

\begin{align}\label{eqn:limit}
\lim_{n \to \infty} \frac{T_{n+1-i}^2}{T_{2n}} \stackrel{(\ref{eqn:simplifyfrac})}{=} \lim_{n \to \infty} \alpha \frac{\phi^{-4i} + \phi^{-8n + 4i} - 2\phi^{-4n}}{\phi^{-2} - \phi^{-8n + 2}} = \alpha \frac{\phi^{-4i}}{\phi^{-2}} = \alpha \frac{\phi^2}{\phi ^ {4i}}.
\end{align}

Now, we get the following 
\begin{align}\label{eqn:homogenval}
\lim_{n \to \infty} \sum_{i = 1}^{n - 1} \frac{(2i + 1)T_{n+1-i}^2}{T_{2n}} &\overset{(\ref{eqn:dominatedconv}), (\ref{eqn:limit})}{=} \alpha \phi^2 \sum_{i = 1}^{\infty} \frac{2i + 1}{\phi^{4i}} \\ \nonumber
&= \frac{\alpha \phi^2(3 \phi^{-4} - \phi^{-8})}{(1 - \phi^{-4})^2} \\ \nonumber
&= \frac{10 + 3 \sqrt{5}}{25} = 0.668328.
\end{align}

Finally, observing that
\[\lim_{n \to \infty} \frac{2n - 1}{T_{2n}} = 0,\]
we get
\[\lim_{n \to \infty} \frac{S_{2n}}{T_{2n}} =  \lim_{n \to \infty} \frac{2n - 1}{T_{2n}} + \lim_{n \to \infty} \sum_{i = 1}^{n - 1} \frac{(2i + 1)T_{n+1-i}^2}{T_{2n}} = 0.668328.\]
This completes our proof.

\qed

\begin{remark}
It is worth noting that if $\Gamma_{2n}$ has $\sigma_i = 1$ for all $i$, this limit already agrees with $\dfrac{S_{2n}}{T_{2n}}$ in three decimal places for $n \ge 6$. So, for a fixed $n$, we can approximate the proportion of balanced spanning trees {in triangular lattices where $\sigma_i = 1$ for all $i$} very accurately with this limit.
\end{remark}

\subsection{Finding the lower bound}

The main goal of this section is to prove the following two propositions: 

\begin{proposition}\label{prop:lowerboundeven}
For any even $n \geq 1$, if $\sigma_i = 1$ for all $i \in \{1, 2, \dots, m-1\}$, the following holds:

\[\lim_{n\to \infty}\frac{S_{2n}}{T_{2n}} = 0.63087.\]

\end{proposition}.

\begin{proposition}\label{prop:lowerboundodd}
For any odd $n \geq 1$, if $\sigma_i = 1$ for all $i \in \{1, 2, \dots, m-1\}$, the following holds:

\[\lim_{n\to \infty}\frac{S_{2n}}{T_{2n}} = 0.54493.\]

\end{proposition}.

\begin{proof}[Proof for Proposition \ref{prop:lowerboundeven}]
    By \Cref{thm:Snforheterogeneous}, if $n = 2m$, we get, 
\[ S_{2n}= 2n-1+ 3T_{2m}^2 +\sum_{i=1}^{m-1}[(8i+4)T_{2(m-i)+1}T_{2(m-i)}] .\]

Observe that, we have, 

\begin{align*}
    \frac{(8i + 4)T_{2(m - i) + 1}T_{2(m - i)}}{T_{4m}}  &\stackrel{(\ref{eqn:inequality})}{<} \frac{8i + 4}{9T_{4i - 1}} \\
    &= \frac{8i + 4}{9\alpha(\phi^{8i - 3} - \phi^{-(8i - 3)})} \\
    &< \frac{8i + 4}{9\alpha(\frac{3}{2})^{8i - 3}} \\
    &= \frac{\sqrt{5}}{9}(8i + 4)\bigg(\frac{2}{3}\bigg)^{8i - 3}.
\end{align*}

With this fact, we can use the dominated convergence theorem to establish the following: 
\begin{equation}\label{eqn:dominatedconv2}
    \lim_{m \to \infty} \sum_{i = 1}^{m - 1} \frac{(8i + 4)T_{2(m - i) + 1}T_{2(m - i)}}{T_{4m}} = \sum_{i = 1}^{\infty} (8i + 4) \lim_{m \to \infty} \frac{T_{2(m - i) + 1}T_{2(m - i)}}{T_{4m}}.
\end{equation}
Observe that 

\[\lim_{m \to \infty} \frac{T_{2(m - i) + 1}T_{2(m - i)}}{T_{4m}} = \lim_{m \to \infty} \alpha\frac{\phi^{-8i - 2} - \phi^{-8m + 2} - \phi^{-8m - 2} + \phi^{-8(2m - i)+2}}{\phi^{-2} - \phi^{-(16m - 2)}} = \frac{\alpha}{\phi^{8i}}.\]

Now, with this, we get 
\begin{equation}\label{eqn:firstpart}
\begin{aligned}
    \lim_{m \to \infty} \sum_{i = 1}^{m - 1} \frac{(8i + 4)T_{2(m - i) + 1}T_{2(m - i)}}{T_{4m}} 
    &\stackrel{(\ref{eqn:dominatedconv2})}{=} \sum_{i = 1}^{\infty} \alpha\frac{8i + 4}{\phi^{8i}} \\ 
    &= \alpha\bigg[\frac{8\phi^{-8}}{(1 - \phi^{-8})^2} + \frac{4\phi^{-8}}{1 - \phi^{-8}}\bigg] \\ 
    &= 0.11841.
\end{aligned} 
\end{equation}
Now, also observe that,
\begin{equation}\label{eqn:secondpart}
    \lim_{m \to \infty} \frac{3T_{2m}^2}{T_{4m}} \stackrel{(\ref{eqn:unifiedeqn})}{=} \lim_{m \to \infty} 3 \alpha \frac{\phi^{-4} + \phi^{-(16m - 4)} - 2\phi^{-8m}}{\phi^{-2} - \phi^{-(16m - 2)}} = \frac{3\alpha}{\phi^2} = 0.51246.
\end{equation}

So, we obtain, 

\[\lim_{m \to \infty} \frac{S_{4m}}{T_{4m}} = \lim_{m \to \infty} \frac{4m - 1}{T_{4m}} + \lim_{m \to \infty}\frac{3T_{2m}^2}{T_{4m}} + \lim_{m \to \infty} \sum_{i = 1}^{m - 1} \frac{(8i + 4)T_{2(m - i) + 1}T_{2(m - i)}}{T_{4m}} \stackrel{(\ref{eqn:firstpart}), (\ref{eqn:secondpart})}{=} 0.63087.\]
\end{proof}

\begin{proof}[Proof for Proposition \ref{prop:lowerboundodd}]
By Theorem \ref{thm:Snforheterogeneous}, we establish that
\[\lim_{n \to \infty} \dfrac{S_{2n}}{T_{2n}}=\lim_{n \to \infty}\bigg(\dfrac{2n-1}{T_{2n}}+\sum_{i=0}^{m-1}\dfrac{(8i+8)T_{2(m-i)+1}T_{2(m-i)}}{T_{2n}}\bigg).\]
Again, since 
\[\lim_{n \to \infty}\dfrac{2n-1}{T_{2n}}=0,\]
we can eliminate the first term. From (\ref{eqn:inequality}), we get
\[9T_{2(m-i)+1}T_{2(m-i)}T_{4i+1} < T_{4m+2}=T_{2n},\] and consequently,
\[\dfrac{(8i+8)T_{2(m-i)+1}T_{2(m-i)}}{T_{2n}} < \dfrac{8i+8}{T_{4i+1}}.\]
Therefore, we can use the dominated convergence theorem to derive the following:
\[\lim_{n \to \infty} \dfrac{S_{2n}}{T_{2n}}=\sum_{i=0}^{\infty}(8i+8)\lim_{n \to \infty}\dfrac{T_{2(m-i)+1}T_{2(m-i)}}{T_{2n}}.\]

From \ref{eqn:unifiedeqn}, we can simplify the equation to

\begin{equation}\label{eqn:oddfinal}
    \lim_{n \to \infty} \dfrac{S_{2n}}{T_{2n}}=\sum_{i=0}^{\infty}(8i+8)\alpha \phi^{-8i-4}=\dfrac{8\alpha \phi^{-4}}{(1-\phi^{-8})^2}=0.544933.
\end{equation}

\end{proof}


\subsection{Random triangular strip lattices} Finally, we aim to determine the expected proportion of balanced spanning trees given a random configuration of the diagonals. 

First, we precisely describe our random model. Let $\Gamma_{2n}$ be a $2$-by-$n$ triangular strip lattice where the diagonal of each square face is chosen independently, such that it connects the top-left and bottom-right vertices with probability $p$, and the top-right and bottom-left vertices with probability $1-p$. As before, let $S_{2n}$ denote the number of balanced spanning trees and $T_{2n}$ denote the total number of spanning trees in $\Gamma_{2n}$.

\par

We now state a more general version of Main Theorem 2:

\begin{theorem}[Average Proportion for Triangular Strip Lattices]
\label{thm:generalized_expected_proportion}
Let $\Gamma_{2n}$ be a random graph chosen as described above. As $n \to \infty$, the expected asymptotic proportion of balanced spanning trees is given by:

Even Case ($n = 2m$):
\[
    \lim_{n \to \infty}\mathbb{E}\bigg[\frac{S_{2n}}{T_{2n}}\bigg] = \frac{9\sqrt{5} - 15}{10} + (2p^2 - 2p + 1)\frac{95 - 39\sqrt{5}}{50} + (2p - 2p^2)\frac{210 - 82\sqrt{5}}{225}.\]

Odd Case ($n = 2m + 1$):
\[
    \lim_{n \to \infty}\mathbb{E}\bigg[\frac{S_{2n}}{T_{2n}}\bigg] = (2p^2 - 2p + 1)\frac{10 + 3\sqrt{5}}{25} + (2p - 2p^2)\frac{68 + 28\sqrt{5}}{225}.
\]
\end{theorem}

\begin{proof}
   Recall that $\sigma_i = 1$ if the $i$th and $(n  - i)$th square have the same type of diagonal and $\sigma_i = 0$ otherwise. With standard calculations, we observe the following: 
\[
\mathbb{P}(\sigma_i) = 
\begin{cases} 
p^2 + (1 - p)^2 & \text{when } \sigma_i = 1\\ 
2p(1 - p) & \text{when } \sigma_i = 0 .
\end{cases}
\]
From this, we see that, 
\begin{equation}\label{eqn:expectedsigma}
    \mathbb{E}[\sigma_i] = p^2 + (1 - p)^2 = 2p^2 - 2p + 1
    \quad \text{and} \quad
    1 - \mathbb{E}[\sigma_i] = 2p - 2p^2.
\end{equation}

With this, we proceed to the calculation of the expected value of $\frac{S_{2n}}{T_{2n}}$ as $n$ limits to infinity. We deal separately with the odd and even cases. 

\begin{proof}[Even Case] 
    For $n = 2m$ using \Cref{thm:Snforheterogeneous}, (\ref{eqn:expectedsigma}), and linearity of expectation, we have the following: 

\begin{equation}\label{eqn:expectedvalfrac}
    \begin{split}
\mathbb{E}\bigg[\frac{S_{2n}}{T_{2n}}\bigg] = \frac{4m - 1}{T_{4m} }+ \frac{3T_{2m}^2}{T_{4m}} +\sum_{i=1}^{m-1}[(2p^2 - 2p + 1)((4i+1)T_{2(m-i)+1}^2+(4i+3)T_{2(m-i)}^2) \\
+(2p - 2p^2)(8i+4)T_{2(m-i)+1}T_{2(m-i)}].
\end{split}
\end{equation}
So, from (\ref{eqn:expectedvalfrac}), (\ref{eqn:homogenval}), (\ref{eqn:secondpart}), and (\ref{eqn:firstpart}) and standard limit calculations, we obtain the following: 

\[\lim_{n \to \infty}\mathbb{E}\bigg[\frac{S_{2n}}{T_{2n}}\bigg] = \frac{9\sqrt{5} - 15}{10} + (2p^2 - 2p + 1)\frac{95 - 39\sqrt{5}}{50} + (2p - 2p^2)\frac{210 - 82\sqrt{5}}{225}.\]

\end{proof}

\begin{proof}[Odd Case] 
    For $n = 2m + 1$ using \Cref{thm:Snforheterogeneous}, (\ref{eqn:expectedsigma}), and linearity of expectation, we have the following:

\begin{equation}\label{eqn:expectedvalfracodd}
    \begin{split}
\mathbb{E}\bigg[\frac{S_{2n}}{T_{2n}}\bigg] = \frac{4m + 1}{T_{4m + 2} }+\sum_{i=0}^{m-1}[(2p^2 - 2p + 1)((4i+3)T_{2(m-i)+1}^2+(4i+5)T_{2(m-i)}^2) \\
+(2p - 2p^2)(8i+8)T_{2(m-i)+1}T_{2(m-i)}].
\end{split}
\end{equation}
So, from (\ref{eqn:expectedvalfracodd}), (\ref{eqn:homogenval}),  (\ref{eqn:oddfinal}) and standard limit calculations, we obtain the following: 

\[\lim_{n \to \infty}\mathbb{E}\bigg[\frac{S_{2n}}{T_{2n}}\bigg] = (2p^2 - 2p + 1)\frac{10 + 3\sqrt{5}}{25} + (2p - 2p^2)\frac{68 + 28\sqrt{5}}{225}.\]

\end{proof}

A case of particular interest occurs when $p = \frac{1}{2}$, representing the scenario where the two possible diagonals in each square face of the $2 \times n$ grid are equally likely to be selected. In such a case, we obtain the following: 
\par
If $n = 2m$, 

\begin{equation}\label{eqn:finalevenratio}
    \lim_{n \to \infty}\mathbb{E}\bigg[\frac{S_{2n}}{T_{2n}}\bigg] = \frac{9\sqrt{5} - 15}{10} + \frac{95 - 39\sqrt{5}}{100} + \frac{210 - 82\sqrt{5}}{450} = 0.6496.
\end{equation}

Similarly, if $n = 2m + 1$, 

\begin{equation}\label{eqn:finaloddratio}
    \lim_{n \to \infty}\mathbb{E}\bigg[\frac{S_{2n}}{T_{2n}}\bigg] = \frac{10 + 3\sqrt{5}}{50} + \frac{68 + 28\sqrt{5}}{450} = 0.6066.
\end{equation}
It is worth noting that these are just the averages of their respective upper and lower bounds.
\end{proof}

\section{Conclusion and Future Directions}


\par

In this paper, we explored $2 \times n$ triangular strip lattices. Future research could focus on $k \times n$ triangular strip lattices where $k \geq 3$. Unlike the case for $k = 2,$ when $k \ge 3,$ the number of spanning trees in a triangular strip lattice depends on its diagonal configuration. For example,  the triangular strip lattices in Figures \ref{fig:3by3with2146spanningtrees} and \ref{fig:3by3with2080spanningtrees} have the same $k$ and $n$ values, but their diagonal configurations are different. Therefore, they have different numbers of spanning trees ({this can be verified with the help of Kirchhoff's Matrix Tree Theorem}), complicating the computations in \Cref{sec:proportion}.
\begin{figure}[H]
    \centering
    \begin{tikzpicture}[x=2.5cm, y=1.5cm, line join=round, line cap=round]

    \definecolor{nodered}{RGB}{255, 0, 0}
    
    \draw[black, line width=1.5pt] (0,2) -- (2,2);
    \draw[black, line width=1.5pt] (0,1) -- (2,1);
    \draw[black, line width=1.5pt] (0,0) -- (2,0);
    
    \draw[black, line width=1.5pt] (0,0) -- (0,2);
    \draw[black, line width=1.5pt] (1,0) -- (1,2);
    \draw[black, line width=1.5pt] (2,0) -- (2,2);
    
    \draw[black, line width=1.5pt] (0,2) -- (1,1);
    \draw[black, line width=1.5pt] (0,1) -- (1,0);
    \draw[black, line width=1.5pt] (1,1) -- (2,2);
    \draw[black, line width=1.5pt] (1,1) -- (2,0);

    \foreach \x in {0,1,2} {
        \foreach \y in {0,1,2} {
            \fill[nodered] (\x, \y) circle (4.5pt);
        }
    }

\end{tikzpicture}
    \caption{A $3 \times 3$ triangular strip lattice with 2146 possible spanning trees.}
    \label{fig:3by3with2146spanningtrees}
\end{figure}
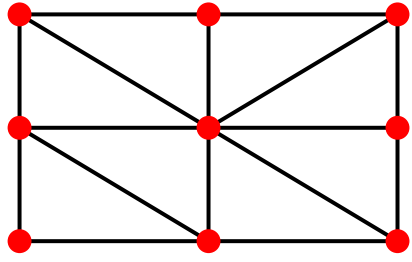
\begin{figure}[H]
    \centering
    \begin{tikzpicture}[x=2.5cm, y=1.5cm, line join=round, line cap=round]

    \definecolor{nodered}{RGB}{255, 0, 0}
    
    \draw[black, line width=1.5pt] (0,2) -- (2,2);
    \draw[black, line width=1.5pt] (0,1) -- (2,1);
    \draw[black, line width=1.5pt] (0,0) -- (2,0);
    
    \draw[black, line width=1.5pt] (0,0) -- (0,2);
    \draw[black, line width=1.5pt] (1,0) -- (1,2);
    \draw[black, line width=1.5pt] (2,0) -- (2,2);
    
    \draw[black, line width=1.5pt] (0,2) -- (1,1);
    \draw[black, line width=1.5pt] (0,1) -- (1,0);
    \draw[black, line width=1.5pt] (1,2) -- (2,1);
    \draw[black, line width=1.5pt] (1,1) -- (2,0);

    \foreach \x in {0,1,2} {
        \foreach \y in {0,1,2} {
            \fill[nodered] (\x, \y) circle (4.5pt);
        }
    }

\end{tikzpicture}
    \caption{A $3 \times 3$ triangular strip lattice with 2080 possible spanning trees.}
    \label{fig:3by3with2080spanningtrees}
\end{figure}
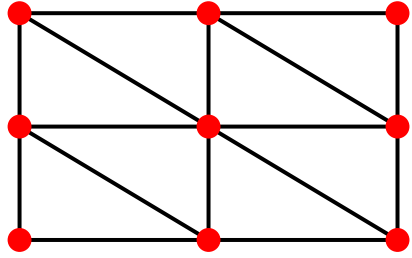
A problem of interest would be to find the diagonal configuration that would result in the maximum and minimum number of spanning trees. In \cite{wang2023triangular}, Wang found an upper and lower bound for the possible number of spanning trees $T_n$ as $n$ varies for the case where all the diagonals in the squares are of the same type (see Figure \ref{fig:3by3with2080spanningtrees} for one of two possibilities for $n = 3$). 

\[(32.4)^n \leq T_n \leq (81.3)^n.\]

\par 

We now discuss some complications that arise when calculating the number of balanced spanning trees in a particular $k \times n$ triangular strip lattice. For $k = 2$, we used dual loops to count the number of balanced spanning trees. Recall that the first step for doing so involves finding balanced partitions (see Claim~\ref{claim:loops}) of vertices in $k \times n$ grid graphs. For $k = 2$, we had exactly $n$ partitions which were straightforward to enumerate. For $k = 3$, the number of partitions is very complex. Specifically, the number of 
balanced partitions $a_n$, as derived in \cite{dc7c0db6-9a45-301d-9d83-89b4d73201a1}, for a $3 \times 2n$ grid graph is: 
\begin{align*}
a_{n} = -3 + n - n^2 &- \frac{1}{5}[(n - 5)F_{3n + 1} + (2n - 1)F_{3n}]\\
\quad &+ \frac{1}{5}\sum_{m = 0}^n \binom{2(n - m)}{n - m}[(3m + 5)F_{3m + 1} - (4m + 3)F_{3m}]
\end{align*} where $F_k$ is, as before, the Fibonacci sequence.
Due to this being a very complex sum, it is  more difficult to enumerate the partitions which is a crucial step in counting the number of balanced spanning trees in both grid graphs and triangular strip lattices using our methods. 
\bibliographystyle{alpha}
\bibliography{main}

\end{document}